\documentclass[reqno,centertags,12pt]{amsart}
\usepackage{amsmath,amsthm,amscd,amssymb,latexsym,verbatim}

\usepackage{graphicx,epsf,cite}

\usepackage{xcolor}

\usepackage{tikz}
\pgfdeclarelayer{nodelayer}
\pgfdeclarelayer{edgelayer}
\pgfsetlayers{edgelayer,nodelayer,main}
\tikzstyle{arrow}=[draw=black,arrows=-latex]

-\marginparwidth \oddsidemargin -4mm \evensidemargin \oddsidemargin

\newtheorem{theorem}{Theorem}[section]

\newtheorem{lemma}[theorem]{Lemma}

\theoremstyle{definition}

\theoremstyle{remark}

\newcounter{smalllist}

\DeclareMathOperator*{\sgn}{sgn}

\allowdisplaybreaks
\numberwithin{equation}{section}

\newcommand{\lb}{\label}

\newcommand{\supp}{\text{\rm{supp}}}

\newcommand{\beq}{\begin{equation}}
\newcommand{\eeq}{\end{equation}}

\newcommand{\bal}{\begin{align}}
\newcommand{\eal}{\end{align}}
\newcommand{\bals}{\begin{align*}}
\newcommand{\eals}{\end{align*}}

\newcommand{\bbR}{{\mathbb{R}}}
\newcommand{\bbD}{{\mathbb{D}}}

\newcommand{\bbT}{{\mathbb{T}}}

\newcommand{\eps}{\varepsilon}

\newcommand{\tht}{\theta}

\begin{document}
\title[Finite Time Singularity for the Generalized SQG Equation]
{Local Regularity and Finite Time Singularity \\ for the Generalized SQG Equation on the Half-Plane}

\author{Andrej Zlato\v s}
\address{\noindent Department of Mathematics \\ UC San Diego \\ La Jolla, CA 92093, USA \newline Email:
zlatos@ucsd.edu}


\begin{abstract}
We show that the generalized SQG equation with $\alpha\in(0,\frac 14]$ is locally well-posed on the half-plane in spaces of bounded integrable locally Lipschitz functions that are natural for its dynamic on domains with boundaries, and allow for some power growth of the derivative in the normal direction at the boundary.  We also show existence of solutions with smooth initial data that exhibit finite time blow-up  in the  whole local well-posedness parameter regime $\alpha\in(0,\frac 14]$, which is the first finite time singularity result for equations (as opposed to patch models) of this type.  Moreover, we prove sharpness of both these results by showing ill-posedness of the PDE in the above spaces when $\alpha>\frac 14$.
\end{abstract}

\maketitle

\section{Introduction and Main Results} \lb{S1}

The 2D Euler and (inviscid) SQG equations of fluid dynamics are both active scalar PDE
\begin{equation} \lb{1.1}
\partial_{t} \tht + u\cdot\nabla \tht =0
\end{equation}
with the Biot-Savart law
\begin{equation} \lb{1.1a}
u := -\nabla^\perp (-\Delta)^{-1+\alpha}\tht,
\end{equation}
where $\alpha=0$ in the Euler case and $\alpha=\frac 12$ in the SQG case, $\nabla^\perp:=(\partial_{x_2},-\partial_{x_1}),$ and $\Delta$ is the Dirichlet Laplacian if the PDE is stated on a domain with a boundary.  When $\alpha\in(0,\frac 12)$, \eqref{1.1} is called the generalized SQG equation (g-SQG), and this family interpolates between Euler and SQG.  The Euler equation was first written down by Euler in the 1750s, while SQG and g-SQG appeared in mathematics and physics literature in the last half-century, in works going back to Blumen \cite{Blu}, Pedlosky \cite{Ped}, and Constantin, Majda, and Tabak \cite{CMT} for SQG, and Pierrehumbert, Held, and Swanson \cite{PHS}, Smith et al.~\cite{Smith}, and Constantin, Iyer, and Wu \cite{CIW} for g-SQG.  

The Euler equation is globally well-posed on both $\bbR^2$ and domains with (smooth enough) boundaries.  For SQG and g-SQG  on $\bbR^2$, local well-posedness results in appropriately regular spaces go back to  Wu \cite{Wu} and Chae, Constantin, and Wu \cite{CCW}, while the question of global well-posedness vs.~blow-up remains one of the central open problems in fluid dynamics; see also Resnick \cite{Resnick} for existence of global weak $L^2$ solutions, Buckmaster, Shkoller, and Vicol \cite{BSV} for non-uniqueness of weak solutions (via convex integration techniques), and C\' ordoba and Mart\' inez-Zoroa \cite{CorMar} for ill-posedness results in $C^k$ for $k\ge 2$ (the last three papers all involve SQG).
  But so far, no well-posedness theory was developed for the SQG and g-SQG equations on domains with boundaries, other than by Constantin and Nguyen for SQG and regular enough solutions that vanish at the boundary  \cite{ConNgu}.
  In fact, only existence of weak  solutions for general $L^2$ initial data was obtained by Constantin and Nguyen \cite{ConNgu2} and by Constantin, Ignatova, and Nguyen \cite{CIN} for SQG, and by Nguyen \cite{Nguyen} for g-SQG, but such solutions might not be unique  (these four papers all considered bounded domains with smooth boundaries).


The reason behind this limited progress on domains with boundaries may be underpinned by the fact that while the Biot-Savart law is non-local for all $\alpha$, the relationship between the solution and its (time-dependent)  stream function $\psi:=-(-\Delta)^{-1+\alpha}\tht$, satisfying $\psi\equiv 0$ on the boundary, is local in one direction  when $\alpha=0$ (i.e.,  $\tht=\Delta\psi$).  This allows for $\psi$, and hence also for $u=\nabla^\perp\psi$ to be regular up to the boundary in the Euler case, as long as the boundary itself is regular (cf.~\cite{HanZla, HanZla2} and references therein), even if $\tht$ does not vanish there.

This property does not carry over to $\alpha>0$, 
and thus complicates development of the corresponding theory on domains with boundaries.  In fact, it is easy to see that when $\tht$ does not vanish on the boundary, its associated velocity $u$ will only be H\" older continuous there, with the normal derivative of the tangential component of $u$ growing as distance to the boundary to the power $-2\alpha$ (see \eqref{2.3} and \eqref{3.2} below). 
As a result, any such theory for \eqref{1.1} has to involve solutions whose regularity deteriorates at the boundary in some sense.  

In contrast, Kiselev, Yao, and the author did establish in 2016 local well-posedness for the corresponding patch problem on the half-plane when $\alpha\in(0,\frac 1{24})$, where patches with $H^3$ boundaries are allowed to touch the domain boundary \cite{KYZ} (Gancedo and Patel recently proved this for $H^2$ patches and  $\alpha\in(0,\frac 16)$ \cite{GanPat}).  In obtaining this result, they leveraged the fact that angles of the tangent lines of smooth enough patch boundaries must vanish at the domain boundary, so the lower regularity of the tangential velocity at that boundary will not distort the patches critically, at least for small enough $\alpha>0$.  Hence even though the velocity $u$ does not remain smooth up to the boundary, the patches themselves do stay regular, so \cite{KYZ} could avoid having to study solutions with singular boundary behaviors in this setting.  

The first main contribution of the present paper is the proof that there is in fact a family of spaces in which local well-posedness does hold for \eqref{1.1} with $\alpha\in(0,\frac 14]$.  Moreover,  this family is natural for the PDE in the sense that it captures the boundary growth of derivatives of solutions dictated by the dynamic of the equation --- and hence also captures the level of boundary regularity exhibited by solutions with {\it up-to-the-boundary smooth} initial data.  
We do this on the half-plane $\bbR\times\bbR^+$ here, where \eqref{1.1a} becomes
\begin{equation} \lb{1.2}
u(t,x):= \int_{\bbR\times\bbR^+} \left( \frac{(x-y)^\perp}{|x-y|^{2+2\alpha}} -
\frac{(x-\bar y)^\perp}{|x-\bar y|^{2+2\alpha}} \right) \tht(t,y) dy,
\end{equation}
with $\bar y:=(y_1,-y_2)$ and $y^{\perp}:=(y_2,- y_1)$ (the latter will be more convenient for us than the more standard definition $y^\perp:=(-y_2, y_1)$, and is equivalent via multiplying $\tht$ by $-1$).  We also dropped some factor $c_\alpha>0$ in \eqref{1.2}, which can be done via appropriate scaling in time.

Our crucial observation is that in order to obtain local regularity of solutions, one should allow {\it only} the normal derivative of the solution to deteriorate at the boundary, while requiring its tangential derivative to remain bounded.
Our key Lemma \ref{L.2.1} below shows that this then results in only the normal derivative of the tangential component of the velocity to also deteriorate at the  boundary (which in turn only affects the normal derivative of the solution), while the tangential derivative of the velocity as well as the normal derivative of its normal component will remain bounded.  In fact, we show that the tangential derivative of the normal component of the velocity then even vanishes at the boundary at an appropriate rate, despite the gradient of the solution (and of the velocity) being unbounded there.  This is again dictated by the dynamic of the PDE, and plays a crucial role in our proof of local well-posedness.

We therefore pick some  $\beta\in[0,1)$ and let $W^{1,\infty}_\beta(\bbR\times\bbR^+)$ be the space of all $\tht:\bbR\times\bbR^+ \to\bbR$ such that
\[
 \|\tht \|_{L^\infty}<\infty \qquad\text{and}\qquad 
 \|\tht\|_{\dot W^{1,\infty}_\beta} := \| \partial_{x_1} \tht \|_{L^\infty}
+ \| \min\{x_2^\beta,1\} \partial_{x_2} \tht \|_{L^\infty} <\infty,
\]
with the norm $ \|\tht\|_{W^{1,\infty}_\beta} :=  \|\tht \|_{L^\infty}+  \|\tht\|_{\dot W^{1,\infty}_\beta}$.
Note that if we define 
\begin{equation} \lb{1.3}
\tilde \tht(x_1,x_2):=\tht \left (x_1, \lambda_\beta(x_2) \right)
\end{equation}
on $\bbR\times\bbR^+$, with 
\begin{equation} \lb{1.3a}
\lambda_\beta(x_2):=
\begin{cases}
 (1-\beta)^{1/(1-\beta)} x_2^{1/(1-\beta)} & x_2\in(0,(1-\beta)^{-1}),\\ 
 x_2- \frac \beta{1-\beta} & x_2\ge (1-\beta)^{-1},
\end{cases}
\end{equation}
then 
\begin{equation} \lb{1.4}
\| \tht\|_{L^\infty} =  \|\tilde\tht\|_{L^\infty}  \qquad\text{and}\qquad 
\| \tht\|_{\dot W^{1,\infty}_\beta} =  \|\tilde\tht\|_{\dot W^{1,\infty}}.
\end{equation}
This means that if we ``stretch'' the strip $\bbR\times[0,1]$ in the vertical direction according to \eqref{1.3}, then  $ W^{1,\infty}_\beta$ becomes precisely $W^{1,\infty}$ (also note that clearly $ W^{1,\infty}_\beta \subseteq W^{1,\infty}_{\beta'}$ when $\beta\le\beta'$).  In these new coordinates, the PDE is again a transport equation  \eqref{2.7}, but with velocity $\tilde u$ from \eqref{2.8} below.  Its dynamic is not incompressible anymore, but it instead preserves the measure $\lambda_\beta'(x_2)dx$ on $\bbR\times\bbR^+$.  Moreover, Lemma~\ref{L.2.1} can be used to show that $\tilde u$ is in fact Lipschitz when $\tilde\tht$ is (see \eqref{2.13}), provided $\beta\in[2\alpha,1-2\alpha]$, which then yields local well-posedness via standard techniques.  This  thus yields our main local regularity result, Theorem \ref{T.1.1}(i) below.  

Having established local regularity, the natural next question is that of global regularity vs.~finite time blow-up.  The 2D Euler equation is known to be {\it critical}:  H\" older \cite{Holder} and Wolibner \cite{Wolibner} showed in 1933 that gradients of smooth solutions on smooth domains cannot grow faster than double-exponentially in time (see also Yudovich \cite{Yudth}), and Kiselev and \v Sver\' ak demonstrated such growth to be possible on domains with a boundary (presence of the boundary was crucial in this work, and while the author showed that gradient of $C^{1,\alpha}$ solutions can grow exponentially on $\bbT^2$ \cite{ZlaEulerexp}, the question of double-exponential gradient growth on $\bbR^2$ or $\bbT^2$, or even of exponential growth of smooth solutions on these domains, remains open).  This suggests $\alpha=0$ to be a borderline case for finite time blow-up, at least for domains with boundaries, and this was confirmed for the corresponding patch problem by Kiselev, Ryzhik, Yao, and the author in \cite{KRYZ}, where they showed finite time blowup for the $H^3$ g-SQG patch model on the half-plane when $\alpha\in(0,\frac 1{24})$, which is the interval of $\alpha$ where \cite{KYZ} established local well-posedness (they also proved global well-posedness for the corresponding Euler problem).  

Their argument in fact equally applies to the PDE \eqref{1.1}, and hence yields finite time blow-up for these $\alpha$ once Theorem~\ref{T.1.1}(i) is proved.  Moreover, the argument  works 
for all $\alpha>0$ such that $\frac {20^{-\alpha}} 6- \frac 1{1-2\alpha} -2^{-\alpha}>0$, which means that it breaks down at $\alpha\approx 0.05$
and is far from covering
the full range of $\alpha$ for which we proved local well-posedness in Theorem~\ref{T.1.1}(i).  Gancedo and Patel used in \cite{GanPat} the same approach with a slightly modified setup that applies up to $\alpha\approx 0.17$, which also  suffices to obtain finite time blow-up for all $\alpha\in(0,\frac 16)$ (i.e., the range of their local regularity result for $H^2$ g-SQG patches on the half-plane), but it does not extend beyond that and so
still does not cover all $\alpha\in(0,\frac 14]$.

We revisit here the argument from \cite{KRYZ} but perform a much more precise analysis, with 
two important improvements.  First, we obtain much smaller errors, and in fact  evaluate the relevant integrals {\it exactly} in the most critical case; and second, we sharpen the estimate on the ``worst case'' scenario for the integrands in those integrals.
  These  allow  us to obtain almost optimal bounds in this argument (we only give up relatively little in the identification of the worst case scenario), thanks to which we are in fact able to capture the full interval $\alpha \in (0,\frac 1{4}]$ on which we obtained local well-posedness.  Near-optimality of our estimates turns out to be crucial because even then our final bound makes the blow-up argument break down at $\alpha\approx 0.257$, barely past the minimum needed for a complete answer to the finite time blow-up question!
See the discussion after Lemma \ref{L.4.2} for more details on this.  

The above two results form the two parts of the following main result of  this paper.

\begin{theorem} \lb{T.1.1}
Let $\alpha\in(0,\frac 14]$, $\beta\in[2\alpha,1-2\alpha]$, and $X_\beta:=W^{1,\infty}_\beta(\bbR\times\bbR^+)\cap L^1(\bbR\times\bbR^+)$.

(i) For any $\tht_0\in X_\beta$, there is  $T_{\tht_0}\in(0,\infty]$ and a unique classical solution $\tht\in L^\infty_{\rm loc}([0,T_{\tht_0});X_\beta)$ to \eqref{1.1} with $\tht(0,\cdot)=\tht_0$, where  $T_{\tht_0}$ is bounded below by some positive function of $(\alpha,\beta,\|\tht_0\|_{X_\beta})$ that is decreasing in the last argument, and $\lim_{t\to T_{\tht_0}} \|\tht(t,\cdot)\|_{\dot W^{1,\infty}_\beta}=\infty$ whenever $T_{\tht_0}<\infty$.

(ii) There is  $\tht_0\in X_\beta\cap C^\infty(\bbR\times\bbR^+)$ such that $T_{\tht_0}<\infty$, so the unique solution $\tht$ blows up in finite time.
\end{theorem}

{\it Remarks.}  
1. Since the dynamic of \eqref{1.1} is measure-preserving, the last claim is equivalent to $\lim_{t\to T_\infty} \|\tht(t,\cdot)\|_{X_\beta}=\infty$.  By \eqref{1.4}, it is also equivalent to $\lim_{t\to T_\infty} \|\tilde \tht(t,\cdot)\|_{\dot W^{1,\infty}}=\infty$.
\smallskip

2. Here we refer to locally Lipschitz solutions for which \eqref{1.1} holds almost everywhere as {\it classical} solutions.  However, it is not difficult to see that the result extends to spaces and initial data with higher degree of (local)  regularity.
\smallskip

3.  One has a choice of $\beta$ when $\alpha\in(0,\frac 14)$, and arguably $\beta=2\alpha$ could be considered most natural as this reflects precisely the  level of boundary regularity exhibited by general solutions with smooth initial data.  
\smallskip

Theorem \ref{T.1.1}(ii) seems to be the first result for this general family of fluid PDEs that demonstrates finite time singularity from smooth initial data. 
We compare it to Elgindi's recent surprising blow-up result for the axisymmetric Euler equation without swirl on $\bbR^3$ \cite{Elgindi},
which requires initial data to have H\"older continuous vorticity (it then becomes unbounded in finite time), something that smooth axisymmetric swirl-free solutions to Euler on $\bbR^3$ cannot achieve because the PDE is globally regular in that setting (this result may be similar in spirit to the exponential growth examples for Euler on $\bbT^2$ in \cite{ZlaEulerexp}).  We also mention here preprints \cite{CheHou, CheHou2}  by Chen and Hou, the second of which was completed shortly after the present work.  These together provide finite time smooth initial data blow-up results for both  3D axisymmetric Euler and 2D Boussinesq equations on bounded domains,
and their proofs heavily employ computer-assisted methods.  

We  note that the proof of Theorem \ref{T.1.1}(ii) suggests that the blow-up there is happening via $\partial_{x_1}\tht (t,\cdot)$ becoming unbounded near the origin (it is a contradiction argument, so it does not actually determine the exact nature of the singularity).
Since functions in $X_\beta$ have $\partial_{x_1}\tht$ bounded up to the boundary, this further illustrates that the blow-up in Theorem \ref{T.1.1}(ii) is truly a non-linear phenomenon, rather than somehow caused by the space $X_\beta$.

Our second main result provides a complete answer to the natural question of sharpness of (both parts of) Theorem \ref{T.1.1}.

 
\begin{theorem} \lb{T.1.2}
PDE \eqref{1.1} is ill-posed in $X_\beta$ for any $\alpha\in(0,\frac 12)$ and $\beta\in[0,1)\setminus [2\alpha,1-2\alpha]$ (with $[2\alpha,1-2\alpha]:=\emptyset$ when $\alpha>\frac 14$). 
\end{theorem}

That is,  for all the couples  $(\alpha,\beta)$ not covered by Theorem \ref{T.1.1}, the PDE \eqref{1.1} is ill-posed in $X_\beta$ (see the start of Section \ref{S3} for details).  The reason behind this lies in Lemma~\ref{L.2.1}, which through the relationship between $u$ and $\tilde u$ in \eqref{2.7a} suggests potential problems with local regularity due to $\partial_{x_2} \tilde u_1$ becoming unbounded near the boundary when $\beta<2\alpha$ and $\partial_{x_1} \tilde u_2$ becoming unbounded there when $\beta>1-2\alpha$.  We show that this indeed results in breakdown of local well-posedness in $X_\beta$ for such $(\alpha, \beta)$.

The proof in fact shows that when $\beta\in[0,2\alpha)$, then the vertical domain stretching \eqref{1.3} is too weak for obtaining well-posedness, while it is also too strong when $\beta\in(1-2\alpha,1).$
This means that not only $X_\beta$ with $\beta\in[2\alpha,1-2\alpha]$ are the right spaces for \eqref{1.1} on the half-plane with respect to the $W^{1,\infty}$ norm, both also that other rates of stretching will not be helpful when it comes to local well-posedness in this norm.

Moreover, we note that since Lemma \ref{L.2.1} shows that $\partial_{x_1} u_1$ is also bounded up to the boundary when $\tht\in X_\beta$, it follows that the lemma (and  hence also Theorem \ref{T.1.1}(i)) cannot possibly extend to $\alpha\ge \frac 12$, because $u_1$ is then unbounded unless $\tht$ vanishes on the boundary.

 
 Finally, our blow-up proof applies identically to the g-SQG patch model from \cite{KYZ, KRYZ}, even though local regularity for that model is presently only known for $\alpha<\frac 1{24}$ (and $\alpha<\frac 1{6}$ in \cite{GanPat}).
 
\begin{theorem} \lb{T.1.3}
For each $\alpha\in(0,\frac 14]$, there are solutions to the $H^3$ gSQG patch model on $\bbR\times\bbR^+$ from \cite{KRYZ, KYZ} (as well as the $H^2$ version from \cite{GanPat}) that cease to exist in finite time.
If the local well-posedness result for this model from \cite{KRYZ, KYZ} (resp.~\cite{GanPat}) holds for some $\alpha\in(0,\frac 14]$, then there exist such solutions 
that exhibit finite time blow-up in the sense of their $H^3$ (resp.~$H^2$) norms diverging to $\infty$ at some finite time.
\end{theorem}

We note that the second claim follows from the first and from the proof of Theorem 1.4 in a recent work of Jeon and the author \cite{JeoZla}.

\bigskip\noindent
{\bf Organization of the Paper and Acknowledgement.}
Our main results, Theorem~\ref{T.1.1}(i), Theorem \ref{T.1.2}, and Theorem \ref{T.1.1}(ii), are proved in Sections \ref{S2}, \ref{S3}, and \ref{S4} in that order.  Sections~\ref{S3} and \ref{S4} are  independent of each other and either of them can be read first.  The author acknowledges partial support by NSF grant DMS-1900943.

\section{Proof of Theorem \ref{T.1.1}(i)} \lb{S2}

Below, dependence of all constants $C_{.}<\infty$ will always be specified in their subscripts (so $C$ will be a universal constant), and the constants are always some finite numbers but may change from line to line.  
The following two lemmas are stated at any fixed time $t$, so we will suppress $t$ in the notation.  

\begin{lemma}  \lb{L.2.0}
If $\tht\in L^1(\bbR\times\bbR^+) \cap L^\infty(\bbR\times\bbR^+)$ and $\alpha\in(0,\frac 12)$, then $u$ from \eqref{1.1a} satisfies  
\beq \lb{2.0}
 \| u \|_{C^{1-2\alpha}} 
 \le C_\alpha  \|  \tht \|_{L^1 \cap L^\infty} 
\eeq
and $\lim_{x\to (s,0)} u_2(x)= 0$ for all $s\in\bbR$.
\end{lemma}

\begin{proof}
Estimate \eqref{2.0} is just Lemma 3.1 in \cite{KRYZ}.  It means that $u$ extends continuously to $\bbR\times\{0\}$, and there we have $u_2\equiv 0$ because the second coordinate of the parenthesis in \eqref{1.2} vanishes when $x_2=0$.
\end{proof}

\begin{lemma}  \lb{L.2.1}
If $\tht\in X_\beta$ with $\alpha\in[0,\frac 12)$ and $\beta\in[0,1)$, then $u$ from \eqref{1.1a} satisfies
\begin{align}
\| \partial_{x_1} u_1 \|_{L^\infty} & \le C_\alpha  ( \| \partial_{x_1} \tht \|_{L^\infty} + \|  \tht \|_{L^1 \cap L^\infty}) , \lb{2.1} \\
\| \max\{x_2^{2\alpha-1},1\}\partial_{x_1} u_2 \|_{L^\infty} & \le C_\alpha    ( \| \partial_{x_1} \tht \|_{L^\infty} + \|  \tht \|_{L^1 \cap L^\infty}),  \lb{2.1a} \\
\| \partial_{x_2} u_2 \|_{L^\infty} & \le C_\alpha  ( \| \partial_{x_1} \tht \|_{L^\infty} + \|  \tht \|_{L^1 \cap L^\infty}), \lb{2.2} \\
\| \min\{x_2^{2\alpha},1\} \partial_{x_2} u_1 \|_{L^\infty} & \le C_{\alpha,\beta}   \big( \| \min\{x_2^\beta,1\} \partial_{x_2} \tht \|_{L^\infty}  + \|  \tht \|_{L^1 \cap L^\infty} \big)  \lb{2.3}. 
\end{align}
\end{lemma}

\begin{proof}
From \eqref{1.2} we have
\begin{equation} \lb{2.4}
u(x)= \int_{\bbR^2}  \frac{(x-y)^\perp}{|x-y|^{2+2\alpha}} \, \tht(y) dy = \int_{\bbR^2}  \tht(x-y) \, \frac{y^\perp}{|y|^{2+2\alpha}}  dy,
\end{equation}
provided we extend $\tht$ oddly onto $\bbR\times\bbR^-$.  Thus for any $x\in\bbR\times\bbR^+$ and  $h\in[0,1]$ we have
\begin{align*}
|u(x+he_1)-u(x)| 
& \le \int_{B_2(0)} \frac{\| \partial_{x_1} \tht \|_{L^\infty} h}{|y|^{1+2\alpha}}  dy 
+ \int_{\bbR^2\setminus B_2(x)} Ch |\tht(y)| dy + \int_{B_{2+h}(x)\setminus B_{2-h}(x)} |\tht(y)| dy \\
& \le C_\alpha  (\| \partial_{x_1} \tht \|_{L^\infty} + \|  \tht \|_{L^1} + \| \tht \|_{L^\infty})h,
\end{align*}
where the first integral is from the second integral in \eqref{2.4}, the second integral is from the first integral in \eqref{2.4} (in it we used that $|\nabla_y \frac{(x-y)^\perp}{|x-y|^{2+2\alpha}}|\le C$ when $y\notin B_1(x)$), and the third integral bounds the error in the evaluation of $u(x+he_1)$  that was introduced by adding them.
This now yields \eqref{2.1} as well as
\begin{equation} \lb{2.4a}
\|\partial_{x_1} u_2 \|_{L^\infty}  \le C_\alpha  (\| \partial_{x_1} \tht \|_{L^\infty} + \|  \tht \|_{L^1 \cap L^\infty}).
\end{equation}

Next assume that $x_2\in(0,1)$ and note that oddness of $\tht$ in $x_2$ yields
\begin{align*}
u_2(x+he_1)-u_2(x) &= \int_{\bbR^2}  (\tht(x-y)- \tht(x+he_1-y)) \, \frac{y_1}{|y|^{2+2\alpha}}  dy\\
& =  \int_{\bbR\times (-\infty,x_2]} (\tht(x-y)- \tht(x+he_1-y)) \left( \frac{y_1}{|y|^{2+2\alpha}} - \frac{y_1}{|y-2x_2e_2|^{2+2\alpha}} \right) dy.
\end{align*}
Hence
\begin{align*}
|u_2(x+he_1) & -u_2(x)| 
 \le   \int_{[-x_2,x_2]^2} \| \partial_{x_1} \tht \|_{L^\infty}  h  \frac{|y_1|}{|y|^{2+2\alpha}}  dy \\
& +  \int_{\bbR\times(-\infty,x_2] \setminus [-x_2,x_2]^2} \| \partial_{x_1} \tht \|_{L^\infty}  h |y| C_\alpha x_2 |y|^{-3-2\alpha} dy  \le C_\alpha \| \partial_{x_1} \tht \|_{L^\infty}  h x_2^{1-2\alpha}.
\end{align*}
This and \eqref{2.4a} imply \eqref{2.1a}.

We used neither information about $\partial_{x_2}\tht$ nor oddness of $\tht$ in $x_2$ in the proof of \eqref{2.1}, which will also be the case in the proof of \eqref{2.2}.  For any $x\in\bbR\times\bbR^+$ (without loss assume $x_1=0$) and  $h\in[0,\min\{1,\|  \tht \|_{L^\infty} \| \partial_{x_1} \tht \|_{L^\infty}^{-1}\}]$ we have
\[
u_2(x+he_2)-u_2(x) 
=  \int_{x+[-2,2]^2}  y_1\left( \frac{1}{|x+he_2-y|^{2+2\alpha}} - \frac{1}{|x-y|^{2+2\alpha}} \right) \tht(y) dy + A_1
\]
for some $|A_1|\le Ch\|\tht\|_{L^1}$.  The integral above equals
\begin{align*}
\int_{x+[0,2]\times[-2,2]}  y_1\left( \frac{1}{|x+he_2-y|^{2+2\alpha}} - \frac{1}{|x-y|^{2+2\alpha}} \right) (\tht(y)-\tht(-y_1,y_2)) dy.
\end{align*}
The parenthesis has the same sign as $y_2-x_2-\frac h2$ so absolute value of this integral is at most
\begin{align*}
& \int_{x+[0,2]\times[-2,2]}  y_1\left( \frac{1}{|x+he_2-y|^{2+2\alpha}} - \frac{1}{|x-y|^{2+2\alpha}} \right) 
{\rm sgn} \left(y_2-x_2-\frac h2  \right) 2M(y_1) dy \\
& = \int_{x+[-2,2]^2}  y_1\left( \frac{1}{|x+he_2-y|^{2+2\alpha}} - \frac{1}{|x-y|^{2+2\alpha}} \right) {\rm sgn} \left(y_1 \left(y_2-x_2-\frac h2 \right) \right) M(y_1) dy,
\end{align*}
where $M(s):=\min\{|s|\, \| \partial_{x_1} \tht \|_{L^\infty}, \| \tht \|_{L^\infty}\}$.
But this is precisely $u_2(x+he_2)-u_2(x) $ when $\tht$ is replaced by
\[
\tht^*(y):=\chi_{x+[-2,2]^2}(y) \,{\rm sgn} \left(y_1 \left(y_2-x_2-\frac h2 \right) \right) M(y_1)
\]
(we will denote $u$ corresponding to $\tht^*$ by $u^*$; also note that $\tht^*$ is not odd in $y_2$).  Since this is constant in $y_2$ between the jumps at $y_2=x_2-2,x_2+\frac h2, x_2+2$, \eqref{2.4} now yields
\begin{align*}
u_2^*(x+he_2)-u_2^*(x) &=  \int_{\bbR^2} ( \tht^*(x-y)- \tht^*(x+he_2-y) ) \, \frac{y_1}{|y|^{2+2\alpha}}  dy \\
& = 2 \int_{[-2,2]\times[-h/2,h/2]} \tht^*(x-y) \, \frac{y_1}{|y|^{2+2\alpha}}  dy +A_2,
\end{align*}
where $|A_2|\le 8h \|\tht^* \|_{L^\infty} \le 8h\| \tht \|_{L^\infty}$.  Let $m:=\min\{  \|  \tht \|_{L^\infty}  \| \partial_{x_1} \tht \|_{L^\infty}^{-1},2\}\ge h$. Then the last integral equals
\begin{align*}
 \int_{[-2,2]\times[-h/2,h/2]} & M(y_1) \, \frac{|y_1|}{|y|^{2+2\alpha}}  dy \\
& \le \int_{B_{h}(0)} \frac{\| \partial_{x_1} \tht \|_{L^\infty} }{|y|^{2\alpha}}  dy 
+ 2h \int_h^{m} \| \partial_{x_1} \tht \|_{L^\infty} \frac{1}{y_1^{2\alpha}} dy_1 
+ 2h \int_{m}^2 \|  \tht \|_{L^\infty} \frac{1}{y_1^{1+2\alpha}} dy_1 \\
& \le C_\alpha \| \partial_{x_1} \tht \|_{L^\infty} h^{2-2\alpha} +  C_\alpha \| \partial_{x_1} \tht \|_{L^\infty} m^{1-2\alpha} h
+  C_\alpha \|  \tht \|_{L^\infty}^{1-2\alpha} \| \partial_{x_1} \tht \|_{L^\infty}^{2\alpha} h\\
&\le  C_\alpha ( \| \partial_{x_1} \tht \|_{L^\infty} + \|  \tht \|_{L^\infty}) h,
\end{align*}
where we used the fact that the last integral is zero if $m=2$.  
We thus obtain
\[
|u_2(x+he_2)-u_2(x) | \le C_\alpha ( \| \partial_{x_1} \tht \|_{L^\infty} + \|  \tht \|_{L^\infty}) h+|A_1|+|A_2|  \le C_\alpha  (\| \partial_{x_1} \tht \|_{L^\infty} + \|  \tht \|_{L^1} + \| \tht \|_{L^\infty})h,
\]
and \eqref{2.2} follows.

It remains to prove \eqref{2.3}.  For any $x\in\bbR\times\bbR^+$ (without loss assume again $x_1=0$) and  $h\in[0,\min\{1,\frac{x_2}3\}]$ we have similarly to the proof of \eqref{2.1},
\begin{align*}
u_1(x+he_2)-u_1(x) &= \int_{[-2,2]^2} (\tht(x+he_2-y)-\tht(x-y))  \frac{y_2}{|y|^{2+2\alpha}}  dy \\ 
&+ \int_{\bbR^2\setminus (x+[-2,2]^2)} \left( \frac{y_2}{|x+he_2-y|^{2+2\alpha}} - \frac{y_2}{|x-y|^{2+2\alpha}} \right) \tht(y) dy +A_3
\end{align*}
for some $|A_3|\le 8h\|\tht\|_{L^\infty}$.  Letting $A_4$ be the second integral above, we have $|A_4|\le Ch\|\tht\|_{L^1}$, so let $I$ be the first integral above and let 
\[
M:= \| \min\{x_2^\beta,1\} \partial_{x_2} \tht \|_{L^\infty}.
\]
If $x_2\ge 3$, then clearly $|I|\le C_\alpha h M$.  Otherwise we can use $2 |x_2+h-y_2| \ge |x_2-y_2|$ for $y_2\notin[x_2,x_2+2h]$ to estimate
\beq \lb{2.5}
|I|  \le \int_{[-2,2]^2} 2hM |x_2-y_2|^{-\beta} \frac{|y_2|}{|y|^{2+2\alpha}}  dy + \int_{[-2,2]\times [x_2,x_2+2h]} 2\|\tht\|_{L^\infty} \frac{|y_2|}{|y|^{2+2\alpha}}  dy.
\eeq
The first integral is no more than
\[
C_\alpha hM \int_{[-2,2] }  |x_2-y_2|^{-\beta} |y_2|^{-2\alpha}  dy_2 \le C_{\alpha,\beta} hM \max\{x_2^{1-2\alpha-\beta},1\}.
\]
The second integral in \eqref{2.5} can be estimated by
\begin{align*}
 2 \|\tht\|_{L^\infty} \left( 4h \int_{x_2}^2 \frac {2x_2}{y_1^{2+2\alpha}} dy_1 + 4hx_2\left(\frac{x_2}3\right)^{-1-2\alpha}  \right)
\le C_\alpha h  \|\tht\|_{L^\infty} x_2^{-2\alpha}.
\end{align*}
So we obtain
\[
|u_1(x+he_2)-u_1(x)| \le  C_{\alpha,\beta} M (1+\min\{x_2,1\}^{-2\alpha+1-\beta}) h + C_\alpha \left( \|\tht\|_{L^\infty} \max\{x_2^{-2\alpha},1\} + \|\tht\|_{L^1} \right) h
\]
and \eqref{2.3} follows.
\end{proof}

Lemmas \ref{L.2.0} and \ref{L.2.1} mean that for $\alpha\in(0,\frac 14]$ and $\beta\in[2\alpha,1-2\alpha]$, solutions to  \eqref{1.1} satisfy the a priori estimate
\beq\lb{2.6}
\frac d{dt} \|\tht(t,\cdot)\|_{X_\beta} \le C_{\alpha,\beta} \|\tht(t,\cdot)\|_{X_\beta}^2.
\eeq
To see this, let $\kappa_\beta(x_2):=\min\{x_2^\beta,1\}$, so that $\lambda_\beta$ from \eqref{1.3a} satisfies
\beq\lb{2.9}
\lambda_\beta'(x_2)=\kappa_\beta(\lambda_\beta(x_2))
\eeq
on $(0,\infty)$.
If  $\tht$ solves \eqref{1.1}--\eqref{1.1a} and we let
\beq\lb{2.8}
\tilde u(t,x_1,x_2):=  \left( u_1 (t,x_1, \lambda_\beta(x_2) ), \kappa_\beta(\lambda_\beta(x_2))^{-1} u_2 (t,x_1, \lambda_\beta(x_2) ) \right),
\eeq
then \eqref{2.9} shows that we have
\beq\lb{2.7}
\partial_{t} \tilde\tht + \tilde u\cdot\nabla \tilde\tht =0
\eeq
on $\bbR\times\bbR^+$.  Note that the dynamic of this PDE preserves the measure $\lambda_\beta'(x_2)dx_1dx_2$.
We now obtain
\beq \lb{2.7a}
\begin{split}
\partial_{x_1} \tilde u_1(t,x_1,x_2) &= \partial_{x_1} u_1(t,x_1,\lambda_\beta(x_2)), \\
\partial_{x_1} \tilde u_2(t,x_1,x_2) &= \kappa_\beta(\lambda_\beta(x_2))^{-1} \partial_{x_1} u_2(t,x_1,\lambda_\beta(x_2)), \\
\partial_{x_2} \tilde u_1(t,x_1,x_2) &= \kappa_\beta(\lambda_\beta(x_2)) \,\partial_{x_2} u_1(t,x_1,\lambda_\beta(x_2)), \\
\partial_{x_2} \tilde u_2(t,x_1,x_2) &= \partial_{x_2} u_2(t,x_1,\lambda_\beta(x_2)) - \frac {\kappa_\beta'(\lambda_\beta(x_2))}{\kappa_\beta(\lambda_\beta(x_2))} u_2(t,x_1,\lambda_\beta(x_2)),
\end{split}
\eeq
where we also have
\beq \lb{2.7b}
\frac {\kappa_\beta'(\lambda_\beta(x_2))}{\kappa_\beta(\lambda_\beta(x_2))} =
\begin{cases}
 \beta \lambda_\beta(x_2)^{-1} & x_2\in(0,(1-\beta)^{-1}),\\ 
 0 & x_2\ge (1-\beta)^{-1}.
\end{cases}
\eeq
Hence Lemma \ref{L.2.1}, the definitions of $\kappa_\beta,\lambda_\beta$, and $ \| \kappa_\beta(x_2) \partial_{x_2} \tht (t,\cdot) \|_{L^\infty}  = \| \partial_{x_2} \tilde \tht (t,\cdot) \|_{L^\infty} $ yield
\beq\lb{2.13}
\|\nabla \tilde u(t,\cdot)\|_{L^\infty} \le C_{\alpha,\beta}  ( \| \nabla \tilde \tht (t,\cdot)\|_{L^\infty} + \|  \tht(t,\cdot) \|_{L^1 \cap L^\infty}).
\eeq
Indeed, this follows using $\beta\in[2\alpha,1-2\alpha]$ and 
\beq \lb{2.14}
\kappa_\beta(\lambda_\beta(x_2))=(1-\beta)^{\beta/(1-\beta)} x_2^{\beta/(1-\beta)} 
\eeq
for $x_2\in [0,(1-\beta)^{-1}]$,
 as well as
\beq \lb{2.15}
| u_2 (t,x)|  \le C_\alpha  ( \| \partial_{x_1} \tht (t,\cdot)\|_{L^\infty} + \|  \tht(t,\cdot) \|_{L^1 \cap L^\infty}) \, x_2,
 \eeq
 which holds by \eqref{2.2} and the last claim in Lemma \ref{L.2.0}.
 
 Since $ \|  \tht(t,\cdot) \|_{L^1}$ and  $\| \tht (t,\cdot)\|_{L^\infty}$ stay constant in time because \eqref{1.1} is a transport equation with divergence-free velocity, \eqref{2.13} shows that
 \[
 \frac d{dt} \| \nabla \tilde \tht(t,\cdot)\|_{L^\infty} \le C_{\alpha,\beta} \| \nabla \tilde \tht(t,\cdot)\|_{L^\infty} \left(  \| \nabla \tilde\tht(t,\cdot)\|_{L^\infty} +  \|  \tht_0 \|_{L^1 \cap L^\infty} \right).
 \]
 This now yields \eqref{2.6}.
 
 Having the a priori estimate \eqref{2.6}, it is  easy to construct local solutions  in $X_\beta$ and show their uniqueness.  One can for instance approximate the Biot-Savart kernel $\frac{y^\perp}{|y|^{2+2\alpha}}$ by a sequence of regularized kernels,
(e.g., by multiplying it by a smooth approximation of $\chi_{\bbR^2\setminus B_{1/n}(0)}$), 
 construct solutions to these regularized equations, and then recover the solution to \eqref{1.1} in the limit  (because \eqref{2.6} equally holds for solutions to all the regularizations).  Uniqueness then follows from \eqref{2.7} and \eqref{2.13}.  
 This proves Theorem \ref{T.1.1}(i) because \eqref{2.6} and constancy of $ \|  \tht(t,\cdot) \|_{L^1}$ and  $\| \tht (t,\cdot)\|_{L^\infty}$  show that the unique solution can only cease to exist at some time $T$ when $\lim_{t\to T_{\tht_0}} \|\tht(t,\cdot)\|_{\dot W^{1,\infty}_\beta}=\infty$, as well as that such a time is bounded below by some function of $\alpha,\beta$, and $\|\tht_0\|_{X_\beta}$ (which can clearly be chosen to be decreasing in the last argument).

 \section{Proof of Theorem \ref{T.1.2}} \lb{S3}
 
 Since the proof of Theorem \ref{T.1.1}(ii) is completely independent of this section, a reader interested in the singularity result can first read Section \ref{S4}.
 
 All constants in this section may depend on $\alpha$.
 We split the analysis into two cases: $\beta\in[0,2\alpha)$ and $\beta\in (1-2\alpha,1)$.
 In the first we provide initial data in $X_\beta$ for which no solutions exist in $L^\infty ([0,T];X_\beta)$ for any $T>0$.  In the second case
we construct a sequence of initial data that converges in $X_\beta$ but no sequence of solutions for these data can converge to a solution for the limiting datum in $L^\infty ([0,T];X_\beta)$ for any $T>0$.
 
 \bigskip\noindent
{\bf Case $\beta\in[0,2\alpha)$.}
For $x\in\bbR\times\bbR^+$, let 
 \[
 \phi_0(x):= \min\left\{ 1, d \big(x,\bbR^2\setminus ( [-3,3]\times[-1,2+(1-\beta)^{-1}]) \big) \right\},
 \]
  with $d$ the distance on $\bbR^2$, and let $\phi_1(x):=\max\{0,1-|x|\}$.  
With some $a_n\in(0,2^{-4n-1}]$ ($n=1,2,\dots$), and 
 with $\lambda_\beta$ from \eqref{1.3a} and $\Lambda_\beta(x):=(x_1,\lambda_\beta(x_2))$, define
\beq \lb{3.0}
 \tilde\tht_0 := \phi_0 + \phi_\infty
 \qquad\text{and}\qquad \tht_0  :=  \tilde \tht_0 \circ \Lambda_\beta^{-1} \in X_\beta,
\eeq
where
\beq \lb{3.1}
\phi_\infty(x):= \sum_{n= 1}^\infty a_n\phi_1\left( \frac{x-(0,2^{-4n})}{a_n} \right).
\eeq

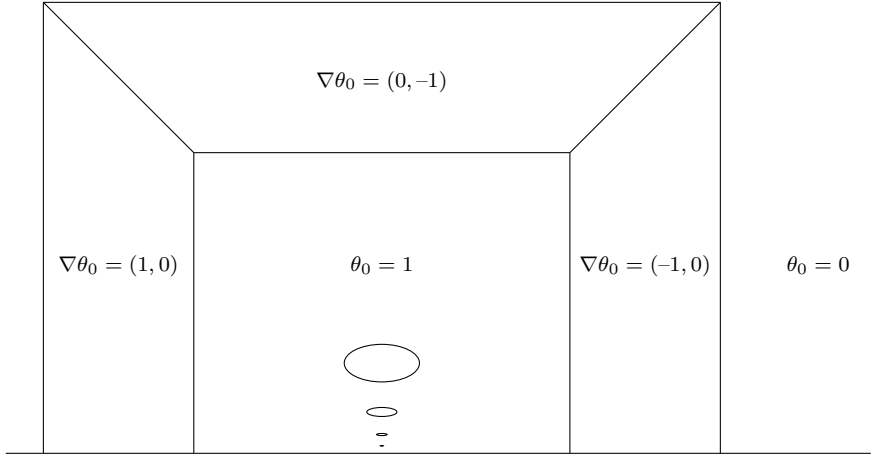
\begin{figure} 
    \pgfdeclarelayer{drawing layer}
    \pgfdeclarelayer{label layer}
    \pgfsetlayers{drawing layer,label layer}
    \begin{tikzpicture}

        \begin{pgfonlayer}{drawing layer}
            \draw (-5.5,0) to (6,0);
            \draw (-3,0) to (-3,4);
            \draw (-3,4) to (2,4);
            \draw (2,0) to (2,4);
            \draw (-5,0) to (-5,6);
            \draw (-5,6) to (4,6);
            \draw (4,0) to (4,6);
            \draw (-3,4) to (-5,6);
            \draw (2,4) to (4,6);

               \node[font=\tiny] (x) at (-4,2.5) {$\nabla \tht_0 = (1,0)$};
               \node[font=\tiny] (x) at (3,2.5) {$\nabla \tht_0 = (\text{--}1,0)$};
               \node[font=\tiny] (x) at (-0.5,4.95) {$\nabla \tht_0 = (0,\text{--}1)$};
               \node[font=\tiny] (x) at (5.3,2.5) {$\tht_0 = 0$};
               \node[font=\tiny] (x) at (-0.5,2.5) {$\tht_0 = 1$};
               
               \draw (-0.5,1.2) ellipse (0.5 and 0.25);
               \draw (-0.5,0.55) ellipse (0.2 and 0.06);
               \draw (-0.5,0.25) ellipse (0.07 and 0.014);
               \draw (-0.5,0.1) ellipse (0.02 and 0.003);

        \end{pgfonlayer}

    \end{tikzpicture}
    \caption{Function $\tht_0$ for $\beta<2\alpha$, with $\phi_\infty \circ \Lambda_\beta^{-1}$ supported on  the ``ellipses''.} \lb{F.3.1}
\end{figure}

It is easy to see that $\tilde \tht_0$ is Lipschitz with constant $1$  because  supports of the summands in the series above are disjoint, and are all contained in $[-1,1]\times(0,1]$.  We will now show that when the $a_n$ are small enough, a classical solution $\tht$ to \eqref{1.1} with initial datum $\tht_0$ would have to instantly leave $X_\beta$ due to $\tilde\tht(t,x):=\tht(t,\Lambda_\beta(x))$ not being Lipschitz on $[0,T]\times\bbR\times\bbR^+$ for any $T>0$.  Specifically, we will essentially obtain $\partial_{x_2} u_1(t,x) \sim x_2^{-2\alpha}$ for small $t,x_2>0$ and $x_1\sim 0$ (with the other three partial derivatives of $u$ sufficiently controlled), which will result in fast horizontal shearing of the ``cap'' functions from the  series in \eqref{3.1}  (and hence instantaneous loss of Lipschitz continuity of $\tilde \tht$).  

As before, extend $\tht_0,\phi_0,\phi_\infty$ oddly onto $\bbR^2$.
Let us first look at the velocity 
\[
v^0(x):= \int_{\bbR^2} \frac{(x-y)^\perp}{|x-y|^{2+2\alpha}}  \phi_0 (\Lambda_\beta^{-1}(y)) dy,
\]
generated by only $\phi_0 \circ \Lambda_\beta^{-1}$.
 Lemma \ref{L.2.1} shows that $\partial_{x_1} v^0_1 $  and $ \partial_{x_2} v^0_2 $ are bounded on $\bbR\times\bbR^+$.  Moreover, from
 \begin{align*}
\partial_{x_1} v^0_2(x) &= - \int_{\bbR^2}  \frac d{dx_1} \phi_0 (\Lambda_\beta^{-1}(x-y)) \, \frac{y_1}{|y|^{2+2\alpha}}  dy \\
 &= \int_{\bbR\times (-\infty,x_2]} \frac d{dx_1} \phi_0 (\Lambda_\beta^{-1}(x-y)) \left(   \frac{y_1}{|y-2x_2e_2|^{2+2\alpha}} - \frac{y_1}{|y|^{2+2\alpha}} \right) dy
 \end{align*}
  and constancy of $\phi_0 \circ \Lambda_\beta^{-1}$  on $[-2,2]\times (0,2]$ and on $[-2,2]\times [-2,0)$ (so the integrand  vanishes when $x\in[-1,1]\times(0,1]$ and $|y|\le 1$) we see that   $x_2^{-1}\partial_{x_1} v^0_2$ is  bounded on $[-1,1]\times (0,1]$.
 
 For the last partial derivative, we can use the same constancy of $\phi_0$ and 
 \[
 v^0_1(x+he_2)-v^0_1(x) = \int_{\bbR^2} \left[ \phi_0 (\Lambda_\beta^{-1}(x+he_2-y))-\phi_0 (\Lambda_\beta^{-1}(x-y)) \right]  \frac{y_2}{|y|^{2+2\alpha}}  dy
 \]
 to see that now instead 
\[
\partial_{x_2} v^0_1(x) - 2\int_{-1}^1\frac{x_2}{|(x_2,y_1)|^{2+2\alpha}}\, dy_1 
 \]
 is bounded on $[-1,1]\times (0,1]$ (because of the jump $\phi_0$ has at the $x_1$ axis).  This means that on this rectangle we have
  \beq \lb{3.2}
\partial_{x_2} v^0_1(x) \ge C^{-1} x_2^{-2\alpha} -  C
 \eeq
 for some constant $C\ge 1$, which we pick so we also have
  \beq \lb{3.3}
\max\left\{ \partial_{x_1} v^0_1 (x), \partial_{x_2} v^0_2(x), x_2^{-1}\partial_{x_1} v^0_2(x) \right\} \le C 
 \eeq
there.  We will now show that when $a_n$ above are small enough, then  these bounds will hold for a short time for the actual velocity $u$, yielding the desired shearing dynamic.

Lemma \ref{L.2.0} applies uniformly to all solutions $\tht$ as above (let us call the constant in it again $C$) because $\|\tht_0\|_{L^1\cap L^\infty}\le \frac {15}{1-\beta}$ regardless of our choice of $a_n$.  So consider  any divergence-free $u$ on  $[0,T]\times\bbR^2$ for some $T$, satisfying $u(t,x_1,-x_2)=(u_1(t,x),-u_2(t,x))$,
\beq \lb{3.4a}
\sup_{t\in[0,T]} \|u(t,\cdot)\|_{C^{1-2\alpha}(\bbR\times\bbR^+)}\le C,
\eeq
$\lim_{x\to (s,0)} u_2(t,x)= 0$ for each $(t,s)\in[0,T]\times \bbR$, and such that the left-hand sides of \eqref{2.1}--\eqref{2.3}  are all finite (with the $L^\infty$ norms taken over $[0,T]\times \bbR\times\bbR^+$).  Note that then $\|x_2^{-1} u_2\|_{L^\infty}$ is also finite, similarly to  \eqref{2.15}, and all these estimates show that $\frac d{dt} z^{t,x}=u(t,z^{t,x})$ and $z^{0,x}:=x$ has a unique solution on $[0,T]$ for each $x\in \bbR^2$ and  $z^{t,\cdot}$ is a measure preserving bijection on $\bbR\times\bbR^+$ (and so also on $\bbR\times\bbR^-$) for each $t\in[0,T]$.  If we let $z^{-t,\cdot}$ be its inverse, then the gSQG velocity 
  \beq \lb{3.3a}
v(t,x):= \int_{\bbR^2} \frac{(x-y)^\perp}{|x-y|^{2+2\alpha}}  \phi_0 (\Lambda_\beta^{-1}(z^{-t,y}))   dy
\eeq
corresponding to  the $u$-transported function $\phi_0 \circ \Lambda_\beta^{-1}$ satisfies
\[
\partial_{x_j} v(t,x) = p.v. \, \int_{\bbR^2}  \partial_{x_j} \frac{(x-y)^\perp}{|x-y|^{2+2\alpha}}  \, \phi_0 (\Lambda_\beta^{-1}(z^{-t,y})) dy .
\]
Since $\|u(t,\cdot)\|_{L^\infty}\le C$, we see that $\phi_0 (\Lambda_\beta^{-1}(z^{-t,y}))=\sgn(y_2)$  for $(t,y)\in[0,\frac 1{2C}]\times[-\frac 32,\frac 32]^2$.    We can then even remove ``p.v.'' in
\begin{align*}
\partial_{x_j} v(t,x) - \partial_{x_j} v^0(x) 
& = \int_{\bbR^2}  \partial_{x_j} \frac{(x-y)^\perp}{|x-y|^{2+2\alpha}}  \left[ \phi_0 (\Lambda_\beta^{-1}(z^{-t,y})) - \phi_0 (\Lambda_\beta^{-1}(y)) \right] dy \\
& = \int_{\bbR^2 \setminus B_{1/2}(x)}  \partial_{x_j} \frac{(x-y)^\perp}{|x-y|^{2+2\alpha}}  \left[ \phi_0 (\Lambda_\beta^{-1}(z^{-t,y})) - \phi_0 (\Lambda_\beta^{-1}(y)) \right] dy
\end{align*}
when $(t,x)\in [0,\frac 1{2C}]\times [-1,1]\times (0,1]$.  From this, \eqref{3.2}, and \eqref{3.3}  we see that after increasing $C$ (it still only depends on $\alpha$), we obtain
  \beq \lb{3.5}
\partial_{x_2} v_1(t,x) \ge C^{-1} x_2^{-2\alpha} -  C 
 \eeq
and
  \beq \lb{3.6}
\max\left\{ \partial_{x_1} v_1 (t,x), \partial_{x_2} v_2(t,x), x_2^{-1}\partial_{x_1} v_2(t,x) \right\} \le C 
 \eeq
for $(t,x)\in [0,\min\{\frac 1{2C},T\}]\times  [-1,1]\times (0,1]$, where the bound on $x_2^{-1}\partial_{x_1} v_2(t,x)$ also uses
\[
\partial_{x_1} v_2(t,x) - \partial_{x_1} v_2^0(x) 
= \int_{\bbR\times (-\infty,x_2]}  \left(  \partial_{x_1} \frac{x_1-y_1}{|x-y|^{2+2\alpha}} -  \partial_{x_1} \frac{x_1-y_1}{|\bar x-y|^{2+2\alpha}} \right)  \left[  \phi_0 (\Lambda_\beta^{-1}(y)) - \phi_0 (\Lambda_\beta^{-1}(z^{-t,y}))\right] dy.
\]


Let us now assume that $\tht\in L^\infty([0,T];X_\beta)$ is  a classical solution to \eqref{1.1} with initial datum given by \eqref{3.0} and \eqref{3.1}, with some $T\in(0,1]$ and some $a_n$ to be determined later.  That is,  $\tilde \tht(t,x):=\tht(t,\Lambda_\beta(x))$ is Lipschitz in $x$ with some constant $L$.  We extend $\tht$  oddly onto $[0,T]\times \bbR^2$, let 
\beq \lb{3.00}
u(t,x):= \int_{\bbR^2} \frac{(x-y)^\perp}{|x-y|^{2+2\alpha}}  \tht(t,y) dy
\eeq
be the velocity corresponding to $\tht$, and define $z^{t,x}$ as above.
Now $\tht(t,x)=\tht_0(z^{-t,x})$ shows that $u=v+w$, where $v$ is from \eqref{3.3a} and 
  \[
w(t,x)
:= \int_{\bbR^2}  \frac{(x-y)^\perp}{|x-y|^{2+2\alpha}} \, \phi_\infty ( \Lambda_\beta^{-1}(z^{-t,y}) )  dy.
\]
Since \eqref{3.5} and \eqref{3.6}
hold by Lemma \ref{L.2.0} and will yield the desired shearing dynamic, we only need to show that $w$ will not distort it too much at small times when the $a_n$ are small.  

From \eqref{2.15} we see that there is  $T_\tht\in(0,\min\{\frac 1{2C},T\}]$ (depending on $\alpha, T,L$) such that
 \beq \lb{3.7}
 z^{t,x}_2x_2^{-1}\in[2^{-1},2] 
 \eeq
 for all $(t,x)\in [0,T_\tht]\times\bbR\times\bbR^+$.
   This, $(\lambda_\beta^{-1})'\ge 1$, and $z^{-t,\cdot}$ being measure-preserving show that there is $C'<\infty$ such that if $x,x'\in \bbR\times (0,\lambda_\beta(2^{1-4N})]$ for some $N\ge 1$, then
\beq \lb{3.10}
|w(t,x)-w(t,x')| \le C' \sum_{n= 1}^{N-1} a_n^3 \lambda_\beta(2^{-4n-3})^{-2-2\alpha} |x-x'| 
+ C' \sum_{n= N}^{\infty} a_n^{2-2\alpha}
\eeq
 for all $t\in[0,T_\tht]$.  This is because  $y_2\ge \lambda_\beta(2^{-4n-2})$  when  $\Lambda_\beta^{-1}(z^{-t,y}) \in B_{2^{-4n-1}}((0,2^{-4n}))$, due to \eqref{3.7} and convexity of $\lambda_\beta$, which means that
 \[
 \min\{ |x-y|, |x'-y| \} \ge \lambda_\beta(2^{-4n-2})-\lambda_\beta(2^{-4n-3})\ge \lambda_\beta(2^{-4n-3})
 \]
for such $y$ when $n\le N-1$, while for $n\ge N$ we simply used the bound
  \[
\left| \int_{\bbR^2}  \frac{(x-y)^\perp}{|x-y|^{2+2\alpha}} \, a_n \phi_1 \left(  \frac{\Lambda_\beta^{-1}(z^{-t,y})-(0,2^{-4n})}{a_n}\right)  dy \right| \le a_n \int_{B_{a_n}(0)} |y|^{-1-2\alpha}dy \le C'' a_n^{2-2\alpha}.
\]
  
  If we pick $a_n:=a 2^{-8n/(1-2\alpha)^2}$ with small enough $a\in(0,(1-\beta)^{1/(1-\beta)}]$ (note that then also $a_n\le \lambda_\beta(2^{-4n-1})$ because $0\le \beta< 2\alpha<1$), then for any $N\ge 1$ we obtain
 \beq \lb{3.7a}
 |w(t,x)-w(t,x')| \le |x-x'| + a_N^{(3-2\alpha)/2}
 \eeq
 whenever $t\in[0,T_\tht]$ and $x,x'\in \bbR\times (0,\lambda_\beta(2^{1-4N})]$ (we decreased the power of $a_N$ here just to remove $C'$).  In fact,  when $N=1$ we can clearly even take any $x,x'\in \bbR\times\bbR^+$.
 
 We will also need a better estimate on $w_2(t,x)-w_2(t,x')$ when $x_2=x_2'$. Oddness of $\phi_\infty ( \Lambda_\beta^{-1}(z^{-t,y}) )$ in $y_2$ shows that the contribution of the $n^{\rm th}$ summand in \eqref{3.1} and of its reflection across the $x_1$ axis  to $\partial_{x_1} w_2(t,x)$ is
 \[
\left| \int_{\bbR^2} \left( \partial_{x_1} \frac{y_1-x_1 }{|x-y|^{2+2\alpha}} - \partial_{x_1} \frac{y_1-x_1 }{|\bar x-y|^{2+2\alpha}}\right) a_n \phi_1 \left(  \frac{\Lambda_\beta^{-1}(z^{-t,y})-(0,2^{-4n})}{a_n}\right)  dy \right|
 \le C'' a_n^3 \lambda_\beta(2^{-4n-3})^{-3-2\alpha} x_2
 \]
 when $n\le N-1$ and $x\in \bbR\times (0,\lambda_\beta(2^{1-4N})]$.  Then, similarly to  \eqref{3.7a}, we obtain
 \[
  |w_2(t,x)-w_2(t,x')| \le x_2 |x-x'| + a_N^{(3-2\alpha)/2}
  \]
when also $x_2=x_2'$ (and $a>0$ is small enough), which combines with \eqref{3.7a} to show that
 \beq \lb{3.11}
 |w_2(t,x)-w_2(t,x')| \le |x_2-x'_2| + \max\{x_2,x_2'\}|x_1-x_1'| +  2a_N^{(3-2\alpha)/2}
 \eeq
  whenever $t\in[0,T_\tht]$ and $x,x'\in \bbR\times (0,\lambda_\beta(2^{1-4N})]$.
 
From \eqref{3.5}, \eqref{3.6}, \eqref{3.7a}, and \eqref{3.11} we see that for any $n\ge 1$, and all $t\in[0,T_\tht]$ and  $x,x'\in [-1,1]\times (0,\lambda_\beta(2^{1-4n})]$ with $x_2\ge x'_2$, we have the bounds
 \beq \lb{3.12}
u_1(t,x) - u_1(t,x')\ge C^{-1}x_2^{-2\alpha}(x_2-x_2') -  3C |x-x'| - a_n^{(3-2\alpha)/2}
\eeq
and 
 \beq \lb{3.13}
|u_2(t,x) - u_2(t,x')| \le 3C |x_2-x'_2| +3Cx_2 |x_1-x_1'| + 2a_n^{(3-2\alpha)/2}.
\eeq
We now take $y_n:=(0,\lambda_\beta(2^{-4n}))$ and $y_n':=(a_n,\lambda_\beta(2^{-4n})-2b_n)$, with $b_n:=2^{-2(2\alpha+\beta)n/(1-\beta)} a_n  $.  
From \eqref{3.7} we see that for $t\in[0,T_\tht]$ we have
\[
\max\{ z^{t,y_n}_2,  z^{t,y_n'}_2 \} \le 2 \lambda_\beta(2^{-4n}) \le \lambda_\beta(2^{1-4n}),
\]
 so \eqref{3.12} and \eqref{3.13} hold with $(x,x')=(z^{t,y_n},  z^{t,y_n'})$ when $t\in[0,T_\tht]$.
 Let 
 \[
 T_n:=\min\{ T_\tht, T_n', T_n''\},
 \]
with $T'_n \ge 0$ the first time when $z^{T'_n,y_n}_1=z^{T'_n,y_n'}_1$  and $T''_n \ge 0$  the first time when 
$z^{T''_n,y_n}_2 - z^{T''_n,y_n'}_2\notin (b_n, 3b_n)$
(we let $T_n':=\infty$ resp.~$T_n'':=\infty$ if there is no such time).
From \eqref{3.12} we see that
 \beq \lb{3.14}
u_1(t,z^{t,y_n}) - u_1(t,z^{t,y_n'})\ge (2C)^{-1} \lambda_\beta(2^{1-4n})^{-2\alpha} b_n \qquad\text{and}\qquad   z^{t,y_n'}_1 - z^{t,y_n}_1\in[0,a_n]
\eeq
when $t\in[0,T_n]$ and $n$ is large enough because $b_n\ge  2C \lambda_\beta(2^{1-4n})^{2\alpha} (3C (a_n+3b_n)+ a_n)$ for large $n$ (note that $b_n\le a_n$).
Then \eqref{3.14} yields
\beq \lb{3.15}
T_n\le a_n 2C \lambda_\beta(2^{1-4n})^{2\alpha}  b_n^{-1} \le  2^{1+1/(1-\beta)} C 2^{-2(2\alpha-\beta)n/(1-\beta)} \to 0,
\eeq
so $T_n<T_\tht$ for large $n$.  Moreover, if $T_n=T_n''$ for large enough $n$, then \eqref{3.13} and \eqref{3.14} yield
\[
T_n \ge   b_n \left( 9Cb_n + 3C \lambda_\beta(2^{1-4n}) a_n +  2a_n^{(3-2\alpha)/2} \right)^{-1} \to (9C)^{-1},
\]
which cannot hold for any large enough $n$ by \eqref{3.15}.  Hence $T_n=T_n'$ for all large $n$, which means that
$|z^{T_n,y_n} - z^{T_n,y_n'}|\le 3b_n$.
Since 
\[
\tilde\tht(T_n,\Lambda_\beta^{-1}(z^{T_n,y_n}))=\tht_0(y_n)=1+a_n \qquad\text{and}\qquad \tilde\tht(T_n,\Lambda_\beta^{-1}(z^{T_n,y_n'}))=\tht_0(y_n')=1,
\]
we see from this and \eqref{3.7} that a Lipschitz constant for $\tilde \tht(T_n,\cdot)$ on $\bbR\times\bbR^+$ is no less than 
\[
a_n \left( 3b_n\lambda_\beta'(2^{-4n-1})^{-1} \right)^{-1} = 3^{-1} (2-2\beta)^{\beta/(1-\beta)} 2^{2(2\alpha-\beta)n/(1-\beta)} \to\infty
\]
for all large $n$.  But this contradicts our hypothesis that
 $\tilde\tht$ is  Lipschitz in $x$ on $[0,T]\times\bbR\times\bbR^+$.
The proof for $\beta\in[0,2\alpha)$ is thus finished. 

\bigskip\noindent
{\bf Case $\beta\in (1-2\alpha,1)$.}
We use here a similar approach as for $\beta<2\alpha$, 
but it will be more complicated due the fact that this time the supports of $\phi_\infty$ and $\nabla \phi_0$ will have to touch.
Since now $\beta\le 1-2\alpha$ does not hold, the bound \eqref{2.1a} will be too weak to control $\nabla\tilde u$, and so the fast shearing dynamic will be happening in the vertical direction.  However, \eqref{2.1a} in fact shows that this dynamic will only be fast in the scaling of $\tilde \tht$, and it will therefore be sufficient (as well as convenient) to let the background $\phi_0$ be Lipschitz on all of $\bbR^2$ in the scaling of $\tht$.   So now let 
 \[
 \phi_0(x):= \min\left\{ 1, d \big( x,\bbR^2\setminus ([-3,0]\times [0,3]) \big) \right\}
 \]
for $x\in\bbR\times\bbR^+$, let again $\phi_1(x):=\max\{0,1-|x|\}$, 
and with some $a_n\in(0,2^{-4n-1}]$ 
let
\beq \lb{3.20}
 \tilde\tht_0 := \phi_0 \circ \Lambda_\beta + \phi_\infty
 \qquad\text{and}\qquad \tht_0 :=  \tilde \tht_0 \circ \Lambda_\beta^{-1} =  \phi_0 + \phi_\infty \circ \Lambda_\beta^{-1} \in X_\beta,
\eeq
where now
\beq \lb{3.21}
\phi_\infty(x):= \sum_{n= 1}^\infty \frac{a_n}n \phi_1\left( \frac{x-(\lambda_\beta(2^{-4n}),2^{-4n})}{a_n} \right).
\eeq

\begin{figure} 
    \pgfdeclarelayer{drawing layer}
    \pgfdeclarelayer{label layer}
    \pgfsetlayers{drawing layer,label layer}
    \begin{tikzpicture}

        \begin{pgfonlayer}{drawing layer}
            \draw (-6,0) to (3.5,0);
            \draw (-3.5,2) to (-3.5,4.5);
            \draw (-3.5,4.5) to (-1,4.5);
            \draw (-1,2) to (-1,4.5);
            \draw (-5.5,0) to (-5.5,6.5);
            \draw (-5.5,6.5) to (1,6.5);
            \draw (1,0) to (1,6.5);
            \draw (-3.5,4.5) to (-5.5,6.5);
            \draw (-1,4.5) to (1,6.5);
            \draw (-3.5,2) to (-5.5,0);
            \draw (-1,2) to (1,0);     
            \draw (-3.5,2) to (-1,2);                   

               \node[font=\tiny] (x) at (-4.5,3.25) {$\nabla \tht_0 = (1,0)$};
               \node[font=\tiny] (x) at (0,3.25) {$\nabla \tht_0 = (\text{--}1,0)$};
               \node[font=\tiny] (x) at (-2.25,5.45) {$\nabla \tht_0 = (0,\text{--}1)$};
               \node[font=\tiny] (x) at (2.3,3.25) {$\tht_0 = 0$};
               \node[font=\tiny] (x) at (-2.25,3.25) {$\tht_0 = 1$};
               \node[font=\tiny] (x) at (-2.25,1) {$\nabla \tht_0 = (0,1)$};
               
               \draw (2.2,1.2) ellipse (0.5 and 0.25);
               \draw (1.55,0.55) ellipse (0.2 and 0.06);
               \draw (1.25,0.25) ellipse (0.07 and 0.014);
               \draw (1.1,0.1) ellipse (0.02 and 0.003);

        \end{pgfonlayer}

    \end{tikzpicture}
    \caption{Function $\tht_0$ for $\beta>1-2\alpha$, with $\phi_\infty \circ \Lambda_\beta^{-1}$ supported on the ``ellipses''.}  \lb{F.3.2}
\end{figure}
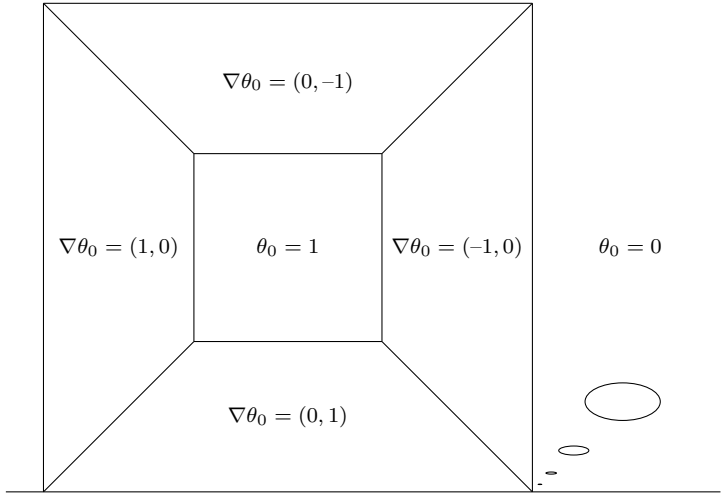

Then $\tilde \tht_0$ is again Lipschitz with constant $1$  because $\lambda_\beta'\le 1$.  We will show that when the $a_n$ are small enough, 
then a solution $\tht\in L^\infty_{\rm loc}([0,T_{\tht_0});X_\beta)$ to \eqref{1.1}  with initial datum $\tht_0$ cannot be a limit of  solutions $\tht_{n_0}$ obtained by truncating the sum in \eqref{3.21} at $n=n_0$  (whose initial data converge to $\tht_0$ in $X_\beta$) because the Lipschitz constants of the functions $\tilde\tht_{n_0}(t,x):=\tht_{n_0}(t,\Lambda_\beta(x))$ on $[0,T]\times\bbR\times\bbR^+$  diverge to $\infty$ for any $T>0$.
The reason for this will now be that $\partial_{x_1} u_2(t,x) \sim x_2^{1-2\alpha}$ for  small $t\ge 0$ and small $x_1\sim x_2>0$  (with the other three partial derivatives of $u$ sufficiently controlled), which together with \eqref{2.7a} will result in $\partial_{x_1} \tilde u_2(t,x)$ becoming arbitrarily large near the origin at small times and hence in fast vertical shearing of the ``cap'' functions from the  series in \eqref{3.21}.

Again extend $\tht_0,\phi_0,\phi_\infty$ oddly onto $\bbR^2$.  Then $\phi_0$ and $\phi_\infty$ are Lipschitz, and so the velocity 
\[
v^0(x):= \int_{\bbR^2} \frac{(x-y)^\perp}{|x-y|^{2+2\alpha}}  \phi_0 (y) dy,
\]
generated by only $\phi_0$ is also Lipschitz with some constant $C$ by the argument in the proof of \eqref{2.1}.
Moreover, the simple form of $\partial_{x_1} \phi_0$ and easy cancellations due to oddness in $x_2$  yield
 \[
\partial_{x_1} v^0_2(x) = - \int_{\bbR^2}  \partial_{x_1} \phi_0 (x-y) \, \frac{y_1}{|y|^{2+2\alpha}}  dy = \int_{x_1}^{x_1+1} \int_{x_1-y_1-x_2}^{x_1-y_1+x_2} \frac{y_1}{|y|^{2+2\alpha}} dy_2 dy_1 + A_1
 \]
 when $x_1,x_2\ge 0$, where $A_1$ is the integral of $\frac{y_1}{|y|^{2+2\alpha}}$ over  the region
 \[
 \{ y_1\in[x_1+2,x_1+3] \,\,\&\,\, y_2 \in [x_1-y_1-x_2,x_1-y_1+x_2] \}
 \]
 minus the integral of $\frac{y_1}{|y|^{2+2\alpha}}$ over
 \[
 \{ y_1\in[x_1,x_1+1] \cup [x_1+2,x_1+3] \,\,\&\,\, y_2 \in [y_1-x_1-3-x_2,y_1-x_1-3+x_2] \}.
 \]
 Since these regions are distance more than 1 from the origin and have total area $6x_2$, we get
 \beq \lb{3.22}
 \partial_{x_1} v^0_2(x)\ge x_2(C^{-1} |x|^{-2\alpha} - C)
 \eeq
for $x\in[0,1]^2$ and some $C\ge 1$.  As we noted above, we also have
  \beq \lb{3.23}
\| \nabla v^0 \|_{L^\infty} \le C.
 \eeq

 Lemma \ref{L.2.0} again applies uniformly to all solutions $\tht$ as above (let us call the constant in it again $C$).
 So consider  any divergence-free $u$ on  $[0,T]\times\bbR^2$ for some $T\in(0,\frac 1{10C}]$ satisfying $u(t,x_1,-x_2)=(u_1(t,x),-u_2(t,x))$,
\beq \lb{3.24}
\sup_{t\in[0,T]} \|u(t,\cdot)\|_{C^{1-2\alpha}(\bbR^2)}\le C,
\eeq
and such that the left-hand sides of \eqref{2.1}--\eqref{2.3}  are all finite (with the $L^\infty$ norms taken over $[0,T]\times \bbR\times\bbR^+$).   So $\|x_2^{-1} u_2\|_{L^\infty}$ is again finite and we can define $z^{t,x}$ as before.  Additionally, we let
 \[
 D_t:=\big\{  x\in \bbR^2 \,\big|\,  |x_2| \ge 3(x_1-z^{t,0}_1)   \big\}
 \]
and assume 
\beq \lb{3.23a}
\| \nabla u \|_{L^\infty \left( \bigcup_{t\in[0,T]} (\{t\}\times D_t) \right)} \le 3C
 \eeq
 as well as 
 \beq\lb{3.42}
\left| u(t,x)- \frac d{dt} z^{t,0} \right| \le 3C|x-z^{t,0}|
 \eeq
 for all $(t,x)\in[0,T]\times\bbR^2$.
Note that this implies 
\beq \lb{3.31}
 (1-3Ct) |x| \le  |z^{t,x}- z^{t,0}|\le (1+6Ct) |x| 
\eeq
for all $(t,x)\in[0,T]\times \bbR^2$, as well as
\beq\lb{3.23c}
z^{t,x}\in \big\{  x\in \bbR^2 \,\big|\,  |x_2| \tan(3Ct) \ge x_1-z^{t,0}_1   \big\} \subseteq D_t
\eeq
 for any $(t,x)\in [0,T]\times \bbR^-_0\times\bbR$ because $T\le \frac 1{10C}$ and the set in \eqref{3.23c}
is the complement of a cone in $\bbR^2$ with vertex at $z^{t,0}$, bisector $[z^{t,0}_1,\infty)\times\{0\}$, and interior half-angle shrinking with speed $3C$. 

Since $\supp\, \phi_0\subseteq \bbR^-_0\times\bbR$, it follows that $\phi_0(z^{-t,\cdot})$ (i.e., $\phi_0$ transported by $u$ for time $t$)  is Lipschitz with constant $\frac 1{1-3Ct} \le \frac 32$ for all $t\in[0,T]$.
Therefore the gSQG velocity 
  \beq \lb{3.24a}
v(t,x):= \int_{\bbR^2}  \phi_0 (z^{-t,x-y}) \,  \frac{y^\perp}{|y|^{2+2\alpha}}  dy
\eeq
corresponding to it  satisfies
\beq \lb{3.24b}
\| \nabla v \|_{L^\infty} \le 2C
 \eeq
(analogously to \eqref{3.23}).  We can also extend \eqref{3.22} to $v$ via
\beq \lb{3.33}
\partial_{x_1} v_2(t,x) - \partial_{x_1} v_2^0(x-z^{t,0}) =  \int_{\bbR^2} \left[ \partial_{x_1}\phi_0 (x-y-z^{t,0}) - \frac d{dx_1} \phi_0 (z^{-t,x-y}) \right]  \frac{y_1}{|y|^{2+2\alpha}} dy ,
\eeq
at least for some  $(t,x)$.
Note that $\supp\,\phi_0$ is composed of 10 trapezoids (2 of them are squares) such that $\nabla \phi_0$ is constant in the interior of each of them, with $|\nabla\phi_0|\in\{0,1\}$.  This and \eqref{3.23a} show that if $z^{-t,x-y}$ with $t\in[0,T]$ lies inside one of these trapezoids and $\partial_{x_1}\phi_0=a$ on that trapezoid, then $ | \frac d{dx_1} \phi_0(z^{-t,x-y}) - a|\le 6Ct$.  
It follows that the absolute value of the parenthesis in \eqref{3.33} is at most $6Ct$ whenever $x-y-z^{t,0}$ belongs to the same trapezoid as $z^{-t,x-y}$.  Also, if  $z^{-t,x-y}\notin \supp \, \phi_0$, then clearly $\frac d{dx_1} \phi_0(z^{-t,x-y})=0$, so the parenthesis is 0 whenever $x-y-z^{t,0}\notin \supp \, \phi_0$.  This means that if we let $g(x-y)$ be that parenthesis when $z^{-t,x-y}$ and $x-y-z^{t,0}$ belong to the same of the above 11 regions $R_1,\dots,R_{11}$ (the 10 trapezoids and  $\bbR^2\setminus \supp\,\phi_0$) and 0 otherwise, then $\|g\|_{L^\infty}\le 6Ct$ and $\|g\|_{L^1}\le 108Ct$ (since $z^{-t,\cdot}$ is measure-preserving) and so oddness of $g$ in $x_2$ shows
\beq \lb{3.34}
\left| \int_{\bbR^2} g(x-y)  \frac{y_1}{|y|^{2+2\alpha}} dy \right| = \left|  \int_{\bbR\times (-\infty,x_2]} g(x-y) \left(  \frac{y_1}{|y|^{2+2\alpha}} - \frac{y_1}{|y-2x_2e_2|^{2+2\alpha}} \right) dy \right| \le C''tx_2^{1-2\alpha}
\eeq
for some $C''>0$, as in the proof of \eqref{2.1a}.

Let us now assume that $z^{-t,x-y}$ and $x-y-z^{t,0}$ do not lie in the same region $R_j$, in which case the parenthesis in \eqref{3.33} is only bounded by 3.
Then \eqref{3.42} and \eqref{3.31} yield
\beq \lb{3.35}
|z^{-t,x-y} - (x-y-z^{t,0})|\le 6Ct|x-y-z^{t,0}|,
\eeq
which means that  $x-y-z^{t,0}$ must be within distance $6Ct|x-y-z^{t,0}|$ of one of the boundaries of the 11 regions (which are 23 straight segments, the rightmost two forming the segment $\{0\}\times [-3,3]$).  If $(t,x)$ is such that $|x_2|\le  2(x_1-z^{t,0}_1)$, then the points $y$ such that $x-y-z^{t,0}$ is within distance $6Ct|x-y-z^{t,0}|$ of one of the boundaries of the 11 regions form no more than $C_1\sqrt{t}$ proportion of $\partial B_r(0)$ for any $r>0$ (with some $C_1>0$).  Using this and the computation that yields \eqref{3.34} shows that the contribution of the points $y$ considered here to the integral in \eqref{3.33}  is no more than $C'''\sqrt t x_2^{1-2\alpha}$ for some $C'''>0$.

The above arguments and \eqref{3.22} show that there is $C'>0$ such that
 \beq \lb{3.36}
 \partial_{x_1} v_2(t,x)\ge x_2 \left[ \left( (2C)^{-1}-C'\sqrt t \right) x_2^{-2\alpha} - C \right]
 \eeq
 holds when $(t,x)\in[0,T]\times[0,1]^2$ and $|x_2|\in [\frac 12(x_1-z^{t,0}_1), 2(x_1-z^{t,0}_1)]$.  This bound and \eqref{3.23} will  play the same role here as \eqref{3.5} and \eqref{3.6} did before.

We now again assume that $\tht\in L^\infty(0,T;X_\beta)$ is  a solution to \eqref{1.1} with initial datum given by \eqref{3.20} and \eqref{3.21}, with some $T\in(0,\frac 1{10C}]$ and some $a_n$ as above.    That is,  $\tilde \tht(t,x):=\tht(t,\Lambda_\beta(x))$ is Lipschitz in $x$ with some constant $L$.  Let us also assume that $\tht$ is the limit of classical solutions $\tht_{n_0}$ with initial data $\tht_{0,n_0}$ obtained by truncating the sum in \eqref{3.21} at $n=n_0$ (we  call that sum $\phi_{n_0}$).  Note that $\tht_{0,n_0}\to\tht_0$ in $X_\beta$.

Fix now any $n_0$, extend $\tht_{n_0}$  oddly onto $[0,T]\times \bbR^2$, and let 
\[
u(t,x):= \int_{\bbR^2} \frac{(x-y)^\perp}{|x-y|^{2+2\alpha}}  \tht_{n_0}(t,y) dy
\]
be the velocity corresponding to $\tht_{n_0}$, and define $z^{t,x}$ as above. (Then $\tht_{n_0}$ is also a unique classical solution to \eqref{1.1} on $\bbR^2$ for at least a short time because $\tht_{0,n_0}$ is Lipschitz. But this might cease to be the case at some point, so we just assume that $\tht_{n_0}\in L^\infty([0,T];X_\beta)$ is some solution.)
Again $\tht_{n_0}(t,x)=\tht_{0,n_0}(z^{-t,x})$ shows that $u=v+w$, where $v$ is from \eqref{3.24a} and 
  \[
w(t,x)
:= \int_{\bbR^2}  \frac{(x-y)^\perp}{|x-y|^{2+2\alpha}} \, \phi_{n_0} ( \Lambda_\beta^{-1}(z^{-t,y}) )  dy.
\]
One can then show as before that if $a_n:=2^{-\gamma n}$ for a large enough $\gamma$, then $w$ will not distort the dynamic provided by $v$ much.  First, \eqref{2.15} again gives \eqref{3.7} for times $t\in[0,T_\tht]$ for some $T_\tht\in(0,T]$, but now we obtain it with interval $[\frac 89, \frac 98]$ instead of $[\frac 12, 2]$.  We then again conclude \eqref{3.11} when $\gamma$ is large enough, and similarly also
\beq \lb{3.40}
|w(t,x)-w(t,x')|\le |x-x'|
\eeq
for all $t\in [0,T_\tht]$ and $x,x'\in D_{t,n_0}:=D_t\cup (\bbR\times [-\lambda_\beta(2^{-4n_0-3}), \lambda_\beta(2^{-4n_0-3})])$ as long as
\beq \lb{3.41}
\supp \, \phi_{n_0} ( \Lambda_\beta^{-1}(z^{-t,\cdot}) ) \subseteq E_{t,n_0}:= \left\{ x\in\bbR^2 \,\bigg |\, |x_2|\in \left[\frac 12(x_1-z^{t,0}_1), 2 (x_1-z^{t,0}_1) \right] \right\}.
\eeq
Here $\gamma, T_\tht$ can be made uniform in $n_0$ because so are all the above estimates.

If now $\tau_{n_0}>0$ is the first time when $\supp \, \phi_{n_0} ( \Lambda_\beta^{-1}(z^{-t,\cdot}) )$ does not lie within $E_{t,n_0}$  
or $\supp \, \phi_{0} ( z^{-t,\cdot} )$ does not lie within $D_{t,n_0}$
($\tau_{n_0}$ must be positive because $n_0<\infty$ and \eqref{3.24} holds), then on the time interval $[0,\min\{\tau_{n_0},T_\tht\}]$ we have \eqref{3.40} when $x,x'\in D_{t,n_0}$, which means that similarly to \eqref{2.6} we can conclude
\[
\frac d{dt} \|\phi_{0} ( z^{-t,\cdot} )\|_{W^{1,\infty}} 
\le C'' \| \phi_{0} ( z^{-t,\cdot} ) \|_{W^{1,\infty}} \big( \| \phi_{0} ( z^{-t,\cdot} ) \|_{W^{1,\infty}}+1 \big).
\]
This is because $\phi_{0} (z^{-t,\cdot} )$ is transported by $u=v+w$, where $v$ is generated by $\phi_{0} ( z^{-t,\cdot} )$ and \eqref{3.40} holds (since $\phi_0$ is Lipschitz on $\bbR^2$, we do not need to use the space $X_\beta$ here).  This $C''$ is independent of $n_0$, so there is $n_0$-independent $T_\tht'\in(0,T_\tht]$ such that \eqref{3.24b} holds for $t\in[0, \min\{\tau_{n_0},T_\tht'\}]$.  But then so do \eqref{3.23a} and \eqref{3.42} because \eqref{3.11} and \eqref{3.40} hold, and \eqref{3.42} means that $\tau_{n_0}\ge T_\tht'$ because $T_\tht'\le T\le \frac 1{10C}$.

Hence the above setup holds on the time interval $[0,T_\tht']$, and in particular \eqref{3.11}, \eqref{3.24b}, and \eqref{3.36} hold on $E_{t,n_0}$ for all $t \in [0,T_\tht']$, and $\supp \, \phi_{n_0} ( \Lambda_\beta^{-1}(z^{-t,\cdot}) )\subseteq E_{t,n_0}$ for all $t \in [0,T_\tht']$.  
We now take 
\[
y:=(\lambda_\beta(2^{-4n_0}),\lambda_\beta(2^{-4n_0})) \qquad \text{and}\qquad 
y':= \left( \lambda_\beta(2^{-4n_0})+ 2n_0^{-2} {a_{n_0}}, \lambda_\beta(2^{-4n_0}-a_{n_0}) \right),
\]
(so now we have $n_0^{-2}a_{n_0}$ in place of $b_n$).  This means that 
\[
y_2-y'_2\approx \lambda_\beta'(2^{-4n_0})a_{n_0} = (1-\beta)^{\beta/(1-\beta)} 2^{-4n_0 \beta/(1-\beta)} a_{n_0} \ll n_0^{-2} {a_{n_0}},
\]
so that, for instance, \eqref{3.11} and \eqref{3.36}  yield for all large $n_0$
\[
|w_2(0,y')-w_2(0,y)| \le C' \left| 2^{-4n_0 \beta/(1-\beta)} a_{n_0} + 2^{-4n_0 /(1-\beta)} n_0^{-2} a_{n_0} + a_{n_0}^{(3-2\alpha)/2} \right| \le  C'' \left| 2^{-4n_0 \beta/(1-\beta)} a_{n_0} \right|
\]
 (provided $\gamma$ is large enough) as well as
\[
v_2(0,y')-v_2(0,y) \ge (4C)^{-1} 2^{-4n_0 (1-2\alpha)/(1-\beta)} n_0^{-2}a_{n_0} .
\]
Similarly to the case $\beta\in[0,2\alpha)$, the above properties of  $v,w$ ensure that this relationship of speeds persists for a while at points $z^{t,y'}$ and $z^{t,y}$, and specifically show that at some time 
\[
T_{n_0} = O \left( 2^{-4n_0 \beta/(1-\beta)} a_{n_0} \left( 2^{-4n_0(1-2\alpha)/(1-\beta)} n_0^{-2}{a_{n_0}} \right)^{-1} \right)
= O \big( {n_0^2} 2^{-4n_0 (\beta+2\alpha-1)/(1-\beta)}  \big) \to 0
\]
we will have 
\[
z_1^{T_{n_0},y'}-z_1^{T_{n_0},y}\in \left[ n_0^{-2} {a_{n_0}} , 3 n_0^{-2} a_{n_0} \right] \qquad\text{and}\qquad z_2^{T_{n_0},y'}=z_2^{T_{n_0},y}
\]
whenever ${n_0}$ is large enough.  But since $\tht_{0,n_0}(y)= n_0^{-1} a_{n_0}$ and $\tht_{0,n_0}(y')=0$, it follows that a Lipschitz constant for $\tilde \tht_{n_0}(T_{n_0},\cdot)$ on $\bbR\times\bbR^+$ is no less than $\frac{n_0}3$.  Since $T_{n_0}<T$ for all large $n_0$, this contradicts our hypothesis that $\tht_{n_0}\to \tht$ in $L^\infty([0,T];X_\beta)$ and the proof is finished.

 \section{Proof of Theorem \ref{T.1.1}(ii)} \lb{S4}
 
The basic setup of our finite time singularity example for each  $\alpha\in(0,\frac 14]$ and $\beta\in [2\alpha,1-2\alpha]$ will be closely related to that from \cite{KRYZ}, where the corresponding patch problem on $\bbR\times\bbR^+$ was studied.  As mentioned in the introduction, a relevant local regularity result in the latter setting was only obtained for $\alpha\in(0,\frac 1{24})$, and the finite time blow-up argument worked on an only  slightly larger interval.
We perform here a much more careful analysis in this setting, which allows us to capture the full range $\alpha\in(0,\frac 14]$, and applies to both \eqref{1.1} and the patch problem (hence it also proves the first claim in Theorem \ref{T.1.3}).  

Let us fix any $\alpha\in(0,\frac 1{4}]$ 
and $\beta\in [2\alpha,1-2\alpha]$ (the particular choice of $\beta$ will be completely irrelevant here), and let $\eps>0$ be a small number to be determined later.
Let $D^+:=\bbR^+\times\bbR^+$, and consider any (smooth if one desires) initial datum $\tht_0\in X_\beta$ that is odd in $x_1$ and satisfies 
\beq \lb{4.0}
\chi_{\Omega'} \le  \chi_{D^+} \tht_0  \le \chi_{\Omega},
\eeq
where $\Omega:=(\eps,3)\times(0,3)$ and $\Omega':=(2\eps,2)\times(0,2)$.  Assume that $T_{\tht_0}=\infty$, and extend the corresponding (unique) solution $\tht$ onto $\bbR^2$ oddly in $x_2$.  Then $\tht$ is odd in both $x_1$ and $x_2$, and transported by $u$ from \eqref{3.00} on the time interval $[0,\infty)$.  Note also that from \eqref{2.1}, \eqref{2.15}, and  $u_1(t,0,s)=0$ for all $(t,s)\in[0,\infty)\times\bbR$ (the latter due to odd symmetry in $x_1$), we have $D^+=\{z^{t,x}\,|\, x\in D^+\}$  for all $t\ge 0$ when $z^{t,x}$ solves $\frac d{dt} z^{t,x}=u(t,z^{t,x})$ and $z^{0,x}:=x$, and hence also
$0\le \sgn(x_1x_2)\, \tht(t,x)\le 1$ for all $(t,x)\in[0,\infty) \times\bbR^2$.  Similarly, for all $t\ge 0$ we have
\beq \lb{4.60}
\Omega_t=\{z^{t,x}\,|\, x\in\Omega_0\} \qquad\text{ and }\qquad D^+\setminus\Omega_t=\{z^{t,x}\,|\, x\in D^+\setminus\Omega_0\}
\eeq
when we let
 \[
 \Omega_t:=\{x\in D^+ \,|\, \tht(t,x)=1\}.
 \]

We will now show that if $\eps>0$ is small enough, then we obtain 
\[
\lim_{t\uparrow T}\| \partial_{x_1} \tht(t,\cdot) \|_{L^\infty(\bbR\times\bbR^+)}=\infty
\]
 for some $T<\infty$, yielding a contradiction with $\tht\in L^\infty_{\rm loc}([0,\infty);X_\beta)$ (which follows from $T_{\tht_0}=\infty$ and Theorem~\ref{T.1.1}(i)) and thus proving the claim.  We will do this by showing that 
$\lim_{t\uparrow T_\eps} d(0,\Omega_t)=0$ for $T_\eps:=25(3\eps)^{2\alpha}$, so that blow-up of $\partial_{x_1}\tht$ no later than by time $T_\eps$ then follows from $\tht(t,0,s)= 0$ for all $t,s>0$.  Finally, this claim will be proved by showing that for all $t\in[0,T_\eps)$ we have $\Omega_t\supseteq L_t$, where
\beq \lb{4.1}
L_t:=\{ x\in D^+ \,|\, x_1\in(X_t,1) \,\,\&\,\, x_2\in(0,x_1)\}
\eeq
is a growing trapezoid with $X_t:=\left[(3\eps)^{2\alpha}- \frac t{25} \right]^{1/2\alpha}$ (and so $X_{T_\eps}=0$).
Note that $X_0=3\eps$, so $\Omega_0 \supseteq \Omega' \supseteq L_0$, and for all $t\in[0,T_\eps]$ we have
\beq \lb{4.2}
\frac d{dt} X_t=-({50 \alpha})^{-1} {X_t^{1-2\alpha}}.
\eeq

Let now  $d(t):= d(D^+\setminus \Omega_t,L_t)$, which is continuous on $[0,T_\eps)$ by \eqref{2.0} and \eqref{4.60}.
So if $\Omega_t\supseteq L_t$ does not hold for some $t\in[0,T_\eps)$, there is a first time $t'\in[0,T_\eps)$ such that $d(t')=0$. Then $t'>0$ because $d(0)\ge\eps$, as well as $L_{t}\subseteq \Omega_{t}$ for all $t\in[0,t']$.  

Let now $\eps'>0$ be arbitrary.  Then \eqref{2.0}, \eqref{4.0}, and $\lim_{\eps\to 0} T_\eps=0$ show that for all small enough $\eps\in(0,\frac{\eps'}3]$ (their interval only depends on $\eps'$ and $\alpha$) we have $\Omega'':=(\eps',\frac 32)\times(0,\frac 32)\subseteq \Omega_t$ for all $t\in[0,T_\eps)$.  This and $d(t')=0$ imply that $\overline {D^+\setminus \Omega_{t'}} \cap \overline{L_{t'}}$ must contain a point from $I_{t'}\cup J_{t'}$, where
\[
I_t:= \{(X_{t},s)\,|\, s\in [0,X_{t}]\} \qquad\text{and}\qquad J_t:= \{(s,s)\,|\, s\in[X_t,\eps']\} 
\] 
(so $I_{t}\cup J_{t}$ is the part of $\partial L_t$ lying outside $\Omega''$, see Figure \ref{F.4.4}).  We will show that if $\eps'>0$ is small enough (depending only on $\alpha$) and $t\in[0,t']$, then $u(t,x)$ points (non-tangentially) outside of $L_t$ for all $x\in J_t$, and $u_1(t,x)\le -(45\alpha)^{-1} X_t^{1-2\alpha}$ for all $x\in I_t$.  This of course makes it impossible for $D^+\setminus \Omega_{t'}$ and $L_{t'}$ to touch at time $t'$, in view of \eqref{4.1} and \eqref{4.2}, yielding a contradiction.

\begin{figure} 
    \pgfdeclarelayer{drawing layer}
    \pgfdeclarelayer{label layer}
    \pgfsetlayers{drawing layer,label layer}
    \begin{tikzpicture}

        \begin{pgfonlayer}{drawing layer}
            \draw (-3,0) to (3,0);
            \draw (-3,0) to (-3,5.2);
            
            \draw (-2.2,0) to (-2.2,0.8);
            \draw (-2.2,0.8) to (1,4);
            \draw (1,4) to (1,0);
            
            \draw[dotted] (-1.2,0) to (-1.2,5);
            \draw[dotted] (-1.2,5) to (2.5,5);
            \draw[dotted] (2.5,5) to (2.5,0);

               \node (x) at (-0.1,1.1) {$L_{t}$};
               \node (x) at (-0.6,4.5) {$\Omega''$};
               \node (x) at (-3,-0.3) {$0$};
               \node (x) at (-1.2,-0.3) {$\eps'$};
               \node (x) at (-2.2,-0.33) {$X_t$};
               \node (x) at (1,-0.3) {$1$};
               \node (x) at (-2.45,0.4) {$I_t$};
               \node (x) at (-1.9,1.55) {$J_t$};
               
        \end{pgfonlayer}

    \end{tikzpicture}
    \caption{Region $L_t$ and segments $I_t$ and $J_t$.} \lb{F.4.4}
\end{figure}
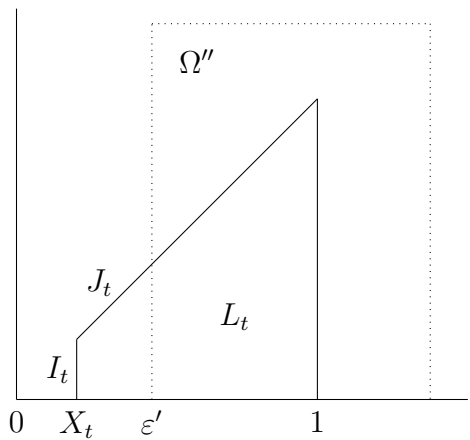

The above is the basic strategy from \cite{KRYZ}, which was there applied with $\Omega_t$ being a solution patch at time $t$ (with $\tht\equiv 1$ on $\Omega_t$ and 0 on $\bbD^+\setminus\Omega_t$, and with a second patch where $\tht\equiv -1$ lying in $\bbR^-\times\bbR^+$, positioned symmetrically to $\Omega_t$).  The crucial estimates on $u(t,\cdot)$ on the set $I_t\cup J_t$ mentioned above
were established in \cite{KRYZ} for only small enough $\alpha>0$, and here we will obtain them for the full local well-posedness range  $\alpha\in(0,\frac 14]$.

\bigskip\noindent
{\bf Estimating the velocity on $I_t$.} 
Let us start with $x=(X_t,x_2)\in I_t$ for some  $t\in[0,t']$ (the case $x\in J_t$ will be  simpler), and we will show that $L_{t}\subseteq \Omega_{t}$ guarantees
\beq\lb{4.3}
u_1(t,x)\le -(45\alpha)^{-1} x_1^{1-2\alpha}
\eeq
as long as $\eps'>0$ was chosen small enough (here of course $x_1=X_t$).  

Odd symmetry of $\tht$ across both axes shows that
\beq\lb{4.4}
u_1(t,x) = \int_{D^+} K_1(x,y) \tht(t,y) dy,
\eeq
where
\beq\lb{4.4a}
K_1(x,y):= \frac{x_2-y_2}{|x-y|^{2+2\alpha}} - \frac{x_2-y_2}{|x-\tilde y|^{2+2\alpha}}
- \frac{x_2+ y_2}{|x-\bar y|^{2+2\alpha}}  + \frac{x_2+  y_2}{|x+ y|^{2+2\alpha}} ,
\eeq
with $\bar y:=(y_1,-y_2)$ and $\tilde y:=(-y_1, y_2)$.  Let also 
\[
A^\pm_x:=\{y\in D^+\,|\, \pm K_1(x,y)>0\}.
\]
Since $0\le\tht\le 1$ on $D^+$, an upper bound on $u_1(t,x)$ can be obtained by replacing $\tht(t,\cdot)$ in \eqref{4.4} by 1 on $A^+_x\cup L'_x$ and by 0 on $A^-_x\setminus L'_x$, for any $L'_x\subseteq L_t$ (because $L_{t}\subseteq \Omega_{t}$).  One should think of $A_x^+$ being the ``bad'' set for an upper bound on $u_1(t,x)$ and $A^-_x$ the ``good'' set, and our task is complicated by the fact that a precise identification of $A^\pm_x$ is  not easy.  We will therefore use a slightly weaker bound, obtained by replacing  $\tht(t,\cdot)$ by 1 on $A^+_x\cup L'_x\cup L''_x$ and by 0 on $A^-_x\setminus (L'_x\cup L''_x)$, as well as $K_1(x,y)$ by
\[
\frac{x_2-y_2}{|x-y|^{2+2\alpha}} - \frac{x_2-y_2}{|x-\tilde y|^{2+2\alpha}} \qquad (\ge\max\{K_1(x,y),0\} \text{ when $y_2\le x_2$})
\]
for all $y\in L''_x$,  with some sets $L'_x\subseteq L_t$ and $L''_x\subseteq \bbR^+\times(0,x_2)$.  
After expanding the integral from \eqref{4.4} onto $\bbR^2$, this upper bound on $u_1(t,x)$  becomes
\beq\lb{4.4b}
\int_{S_x\cup(-S_x)\cup L''_x}  \frac{x_2-y_2}{|x-y|^{2+2\alpha}} dy - \int_{\bar S_x\cup \tilde S_x\cup \tilde L''_x}  \frac{x_2-y_2}{|x-y|^{2+2\alpha}} dy,
\eeq
where $S_x:=(A^+_x\cup L'_x)\setminus L''_x$.  We now make the specific choice
\begin{align*}
L'_x & :=\{ y\in D^+ \,|\, y_1\in(x_1,1) \,\,\&\,\, y_2\in(0,y_1-x_1+x_2)\} \subseteq L_t, \\
L''_x & := \big[ ((0,x_1]\cup[1,\infty))\times(0,x_2) \big] \setminus L'''_x,
\end{align*}
where 
\[
L'''_x:=\{ y\in D^+ \,|\, y_1\in(x_1-x_2,x_1) \,\,\&\,\, y_2\in(0,y_1-x_1+x_2)\}
\]
(recall that $x_2\in[0,x_1]$). These regions are drawn in Figure \ref{F.4.5}, with some being split into several sub-regions for later use.  So, for instance, $L'_x=\bigcup_{j=1}^4 L'_{x,j}$ and $L''_x= L''_{x,1}\cup L''_{x,2}$, where $L''_{x,2}:=[1,\infty)\times(0,x_2)$ lies far to the right and is not pictured.

\begin{figure} 
    \pgfdeclarelayer{drawing layer}
    \pgfdeclarelayer{label layer}
    \pgfsetlayers{drawing layer,label layer}
    \begin{tikzpicture}

        \begin{pgfonlayer}{drawing layer}
            \draw (-6.2,0) to (6.2,0);
            \draw (-2,-2.7) to (-2,3.7);
            \draw (-4,1.5) to (4,1.5);
            \draw (-4,-1.5) to (-4,1.5);
            \draw (0,-1.5) to (0,1.5);
            \draw (2,0) to (2,1.5);
            \draw (4,-1.5) to (4,1.5);
            \draw (-2.5,0) to (-6,3.5);
            \draw (-1.5,0) to (2,3.5);
            \draw (4,1.5) to (6,3.5);
            \draw (-2.5,0) to (-5,-2.5);
            \draw (-1.5,0) to (1,-2.5);
            \draw (2.5,0) to (5,-2.5);
            
            \draw[very thick] (-4,1.5) to (-6,3.5);
            \draw[very thick] (-4,1.5) to (-2,1.5);
            \draw[very thick] (-2,0) to (-2,1.5);
            \draw[very thick] (-2,0) to (-6.2,0);
            \draw[very thick] (-5,-2.5) to (-2.5,0);

            \draw[very thick] (4,1.5) to (6,3.5);
            \draw[very thick] (4,1.5) to (2,1.5);
            \draw[very thick] (2,0) to (2,1.5);
            \draw[very thick] (2,0) to (6.2,0);
            \draw[very thick] (5,-2.5) to (2.5,0);            

            \fill (0,1.5) circle(0.05);
               \node (x) at (-1.8,-0.25) {$0$};
               \node (x) at (-0.15,1.75) {$x$};
               \node (x) at (0.3,-0.25) {$x_1$};
               \node (x) at (2.1,-0.25) {$2x_1$};
               \node (x) at (4.4,-0.25) {$3x_1$};
               \node (x) at (-4.45,-0.25) {$-x_1$};
               \node (x) at (-2.3,1.7) {$x_2$};
               
               \node (x) at (-0.4,0.4) {$L'''_x$};
               \node (x) at (-0.4,-0.4) {$\bar L'''_x$};
               \node (x) at (-3.5,0.45) {$\tilde L'''_x$};
               \node (x) at (-3.5,-0.4) {$-L'''_x$};
               
               \node (x) at (-1.4,0.95) {$L''_{x,1}$};
               \node (x) at (-2.6,1) {$\tilde L''_{x,1}$};
               
               \node (x) at (1,0.7) {$L'_{x,1}$};
               \node (x) at (3,0.7) {$L'_{x,2}$};
               \node (x) at (5.5,1.2) {$L'_{x,3}$};
               \node (x) at (3,2.5) {$L'_{x,4}$};
               
               \node (x) at (1.8,-1.4) {$\bar L'_{x,1}$};
               \node (x) at (3.55,-0.4) {$\bar L'_{x,2}$};
               \node (x) at (5.5,-1.4) {$\bar L'_{x,3}$};
               
               \node (x) at (-5.5,1.2) {$\tilde L'_{x}$};
               \node (x) at (-5.5,-1.4) {$-L'_{x}$};
               
        \end{pgfonlayer}

    \end{tikzpicture}
    \caption{Subregions and reflections of $L'_x,L''_x,L'''_x$ for $u_1$  (not pictured are $L''_{x,2}=[1,\infty)\times(0,x_2)$, $\tilde L''_{x,2}=(-\infty,1]\times(0,x_2)$, and parts of $L'_x,\bar L'_x,\tilde L'_x,-L'_x$).} \lb{F.4.5}
\end{figure}

Our first crucial observation is that
\beq\lb{4.21}
L'''_x\subseteq A_x^+\subseteq L'_x\cup  L''_x\cup L'''_x,
\eeq
which means that $S_x=L'_x \cup L'''_x$ (note that $L'_x, L''_x, L'''_x$ are pairwise disjoint).  The second inclusion in \eqref{4.21} holds because $L'_x\cup  L''_x\cup L'''_x \supseteq \bbR^+\times(0,x_2)$ and for any $y\in \bbR^+\times[x_2,\infty)$ we have
\[
\frac{x_2-y_2}{|x-y|^{2+2\alpha}} \le \frac{x_2-y_2}{|x-\tilde y|^{2+2\alpha}} \qquad\text{and}\qquad
 \frac{x_2+ y_2}{|x-\bar y|^{2+2\alpha}}  \ge \frac{x_2+  y_2}{|x+ y|^{2+2\alpha}}.
\] 
The first inclusion in \eqref{4.21} is due to the following lemma, with $(x,b_1,b_2)$ in the lemma being $(x-y,x_1,x_2)$ here (note that for any $y\in L'''_x$ we have $0<x_1-y_1\le x_2-y_2\le x_2\le x_1$).  We also note that if $x_2=0$, then \eqref{4.21} holds trivially because all the sets in it are empty.

\begin{lemma} \lb{L.4.1}
Whenever $\alpha\ge 0$, $b_1,b_2>0$, and $0< x_1\le x_2< \min\{b_1,b_2\}$, we have
\[
\frac {x_2}{|(x_1,x_2)|^{2+2\alpha}} - \frac {x_2}{|(2b_1-x_1,x_2)|^{2+2\alpha}} 
- \frac {2b_2-x_2}{|(x_1,2b_2-x_2)|^{2+2\alpha}} + \frac {2b_2-x_2}{|(2b_1-x_1,2b_2-x_2)|^{2+2\alpha}} > 0.
\]
\end{lemma}

\begin{proof}
Letting $c:=\frac{b_1}{b_2}$ and $y:=\frac x{b_2}$, this claim becomes equivalent to
\beq \lb{4.6}
\frac {y_2}{|(y_1,y_2)|^{2+2\alpha}} - \frac {2-y_2}{|(y_1,2-y_2)|^{2+2\alpha}} > \frac {y_2}{|(2c-y_1,y_2)|^{2+2\alpha}}  -  \frac {2-y_2}{|(2c-y_1,2-y_2)|^{2+2\alpha}}
\eeq
for $0< y_1\le y_2 < \min\{c,1\}$.  If we let $f(y_1)$ be the left-hand side of \eqref{4.6}, then we need to show $f(y_1)> f(2c-y_1)$.  Direct differentiation of $f$ shows that 
\[
\sgn f'(s)=\sgn \left( (2-y_2)(s^2+y_2^2)^{2+\alpha} - y_2(s^2+(2-y_2)^2)^{2+\alpha} \right),
\]
from which it easily follows that $f'$ vanishes at a single $s'>0$, as well as that $f$ is decreasing on $[0,s']$ and increasing on $[s',\infty)$ (both strictly).
Since $\lim_{s\to\infty} f(s)=0$ and $2c-y_1> y_1$, it suffices to show that $f(y_2)\ge 0$ because then $y_1\le y_2$ implies $f(y_1)> f(s)$ for all ${s> y_1}$.  That is, it suffices to show that
\[
y_2^{-1-2\alpha} - (2-y_2)(y_2^2-2y_2+2)^{-1-\alpha} \ge 0
\]
for all $y_2\in(0,1)$.  Equality holds when $y_2=1$, and a direct computation shows that the derivative of the left-hand side is no more than 
\[
-(1+2\alpha)y_2^{-2-2\alpha} + (y_2^2-2y_2+2)^{-1-\alpha} < 0
\]
for all $y_2\in (0,1)$ (note that $-2y_2+2< 0$).  Hence the claim follows.
\end{proof}

Since $S_x=L'_x \cup L'''_x$, the bound \eqref{4.4b} becomes
\beq \lb{4.4c}
\int_{L'_x\cup(-L'_x)\cup L'''_x\cup(-L'''_x)\cup L''_x}  \frac{x_2-y_2}{|x-y|^{2+2\alpha}} dy - \int_{\bar L'_x\cup \tilde L'_x\cup \bar L'''_x\cup \tilde L'''_x\cup \tilde L''_x}  \frac{x_2-y_2}{|x-y|^{2+2\alpha}} dy.
\eeq
Because the first integrand at $(y_1,y_2)$ equals the second integrand at $(2x_1-y_1,y_2)$, and vice versa, the integrals over the regions contained in $\bbR^-\times\bbR$  cancel with those over the images of these regions under the mapping $y\mapsto (2x_1-y_1,y_2)$ (in the notation from Figure \ref{F.4.5}, region  $\tilde L'_x\cup \tilde L'''_x \cup \tilde L''_{x,1} \cup \tilde L''_{x,2}$ cancels with $L'_{x,3}\cup L'_{x,2} \cup  L''_{x,2}$ and $(-L'_x)\cup (-L'''_x)$ cancels with $\bar L'_{x,3} \cup \bar L'_{x,2}$; these regions are marked by bold lines).  Each of these images is a region contained in the integration domain of the integral not containing the original region (whence the cancelation),  except for the subregions $[(-1,-1+2 x_1)\times(x_2,1)]\cap \tilde L'_x$ and $[(-1,-1+2 x_1)\times(-1,0)]\cap (- L'_x)$, whose images lie in $(1,\infty)\times(\bbR\setminus[0,x_2])$ and so the integrals over them do not get canceled.  However, since $0<x_2\le x_1\le3\eps\ll 1$ and both these regions have areas $\le 2x_1$, the contribution of the integrals over them to \eqref{4.4c} will be bounded above by $16x_1$ and hence negligible in the proof of \eqref{4.3} because $\alpha>0$.  After this cancellation we are left with regions
\begin{align*}
B &:=(0,2x_1)\times(0,x_2), \\
G &:= \{ y\in \bbR^2 \,|\, y_2 > x_2 \,\,\&\,\, y_1\in(y_2-x_2+x_1,y_2-x_2+3x_1) \,\,\&\,\, y_1<1\}, \\
G^- &:= \{ y\in \bbR^2 \,|\, y_2 <0 \,\,\&\,\, y_1\in(-y_2-x_2+x_1,-y_2-x_2+3x_1) \,\,\&\,\, y_1<1\}
\end{align*}
(so $B=L''_{x,1}\cup L'''_{x}\cup L'_{x,1}$,  $G=L'_{x,4}$, and $G^-=\bar L'_{x,1}\cup \bar L'''_{x}$ in Figure \ref{F.4.5}), and this all yields
\[
u_1(t,x)\le \int_{B\cup G}  \frac{x_2-y_2}{|x-y|^{2+2\alpha}} dy - \int_{G^-}  \frac{x_2-y_2}{|x-y|^{2+2\alpha}} dy + 16x_1.
\]

Here the contribution of the ``bad'' region $B$ is positive, while that of the ``good'' regions $G,G^-$ is negative.  To simplify the following computations, we can replace $G,G^-$ by regions $G_*,G^-_*$ that are defined identically but without the constraint $y_1<1$.  This again changes the estimate by less than $C_\alpha x_1$ for some $C_\alpha$ depending only on $\alpha$ (and, in particular, not on $x$), so we obtain the bound
\[
u_1(t,x)\le \int_{B\cup G_*}  \frac{x_2-y_2}{|x-y|^{2+2\alpha}} dy - \int_{G^-_*}  \frac{x_2-y_2}{|x-y|^{2+2\alpha}} dy + (C_\alpha+16)x_1.
\]
The change of variables $z:=\frac{y-x}{x_1}$ turns this bound into
\[
- u_1(t,x) x_1^{2\alpha-1} \ge \int_{G_{b} \cup G^-_{b}}  \frac{|z_2|}{|z|^{2+2\alpha}} dz - \int_{B^-_{b}}  \frac{|z_2|}{|z|^{2+2\alpha}} dz - (C_\alpha+16)x_1^{2\alpha},
\]
where $b:=\frac{x_2}{x_1}$ and the sets $B^-_{b}, G_{b}, G^-_{b}$ (from the following lemma, also drawn in Figure \ref{F.4.7}) are the images of $B,G_*,G^-_*$ under this change of variables.  Estimate \eqref{4.3} now follows for all small enough $\eps'>0$ (recall that $x_1\le3\eps\le\eps'$) by the key Lemma \ref{L.4.2}, which we prove after finishing the proof of Theorem \ref{T.1.1}(ii).

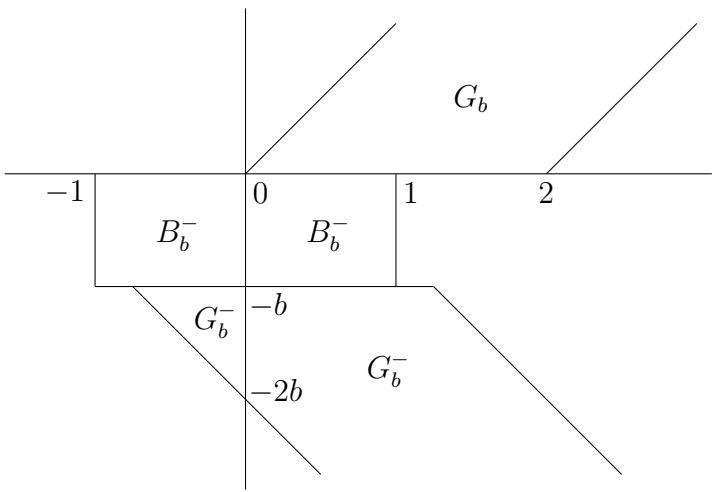
\begin{figure} 
    \pgfdeclarelayer{drawing layer}
    \pgfdeclarelayer{label layer}
    \pgfsetlayers{drawing layer,label layer}
    \begin{tikzpicture}

        \begin{pgfonlayer}{drawing layer}
            \draw (-3.2,1.5) to (6.2,1.5);
            \draw (0,-2.7) to (0,3.7);
            \draw (-2,0) to (2.5,0);
            \draw (-2,0) to (-2,1.5);
            \draw (2,0) to (2,1.5);

            \draw (0,1.5) to (2,3.5);
            \draw (4,1.5) to (6,3.5);
            \draw (-1.5,0) to (1,-2.5);
            \draw (2.5,0) to (5,-2.5);

               \node (x) at (0.2,1.25) {$0$};
               \node (x) at (-2.4,1.25) {$-1$};
               \node (x) at (2.2,1.25) {$1$};
               \node (x) at (4,1.25) {$2$};
               \node (x) at (0.3,-0.25) {$-b$};
               \node (x) at (0.4,-1.4) {$-2b$};
               
               \node (x) at (-0.4,-0.45) {$G^-_b$};
               \node (x) at (1.1,0.7) {$B^-_b$};
               \node (x) at (-0.9,0.7) {$B^-_b$};
               \node (x) at (3,2.5) {$G_b$};
               \node (x) at (1.9,-1.1) {$G^-_b$};

        \end{pgfonlayer}

    \end{tikzpicture}
    \caption{The regions $B^-_{b}, G_{b}, G^-_{b}$ for $u_1$.} \lb{F.4.7}
\end{figure}

\begin{lemma} \lb{L.4.2}
If $\alpha\in(0,\frac 14]$ and for any $b\in[0,1]$ we let  
\begin{align*}
B_b^-& :=(-1,1)\times(-b,0), \\
G_b & :=\{ x\in \bbR^2 \,|\,  x_2>0 \,\,\&\,\,  x_1\in(x_2,x_2+2)\}, \\
G_b^- & :=\{ x\in \bbR^2 \,|\, x_2<-b \,\,\&\,\, x_1\in(-x_2-2b,-x_2-2b+2)\},
\end{align*}
then
\beq \lb{4.7}
\inf_{b\in[0,1]} \left\{ \int_{G_b\cup G_b^-} \frac{|x_2|}{|x|^{2+2\alpha}} dx -   \int_{B_b^-} \frac{|x_2|}{|x|^{2+2\alpha}} dx \right\} \ge \frac 1{40 \alpha}.
\eeq
\end{lemma}

This concludes the case $x\in I_t$. We note that the only place where the above analysis is not optimal is the replacement of $K_1$ by the difference of its first two terms on $L''_{x,1}$ instead of identification of the set $A^+_x\cap L''_{x,1}$ (the set $L''_{x,2}\subseteq[1,\infty)\times\bbR^+$ is irrelevant because its contribution is of lower order and hence negligible here).  
While picking $L'_x=L_t$ for all $x$ may seem to yield further improvement, we in fact do have $L'_x=L_t$ when $x=(X_t,X_t)$ and so $b=1$ above, which will be the most crucial value in the proof of Lemma \ref{L.4.2}.
 
 As for  Lemma~\ref{L.4.2}, if we let $V_{\alpha,b}$ be the value inside the infimum in \eqref{4.7}, then we need to obtain $\inf_{b\in[0,1]} V_{\alpha,b} >0$ for our blow-up analysis to work.  The proof of  Lemma~\ref{L.4.2}  evaluates $V_{\alpha,1}$ {\it exactly} (we could also do it for $V_{\alpha,b}$ with $b\in[0,1)$ but the resulting formulas are quite daunting to deal with, so we instead obtain a much more convenient lower bound) and yields
\[
\frac{2\alpha}{4-2^{-\alpha}} V_{\alpha,1} = \int_0^1 \frac {dq}{(q^2+1)^{\alpha}}  - \frac{2+2^{-\alpha}-2^{1-2\alpha} }{(1-2\alpha) (4-2^{-\alpha})}.
\]
This is decreasing in $\alpha$,  so to cover the full region $\alpha \in(0,\frac 14]$, we certainly need this number to be positive for $\alpha=\frac 14$.  It indeed is, but barely so: it is approximately $ 0.937 - 0.903=0.034$.  
It in fact becomes negative around $\alpha=0.257$,
which illustrates why showing $L'''_x\subseteq A^+_x$ in Lemma~\ref{L.4.1} is also crucial for our analysis to apply to all $\alpha\in(0,\frac 14]$.  If we instead included $L'''_x$ in $L''_x$, this would replace $\bar L'''_{x}$ by $\bar L'_{x,2}$ in the region $G^-$
and ultimately lower the upper end of the interval of $\alpha$ for which the resulting $V_{\alpha,1}$ is positive below $\frac 14$.

\bigskip\noindent
{\bf Estimating the velocity on $J_t$.} 
Now consider $x=(x_*,x_*)\in J_t$ for some  $t\in[0,t']$, so $x_*\in[X_t,\eps']$.  Since 
\[
\Omega_t\supseteq L_t \supseteq L_x:=\{ y\in L_t\,|\, y_1> x_1 \},
\]
the proof from the case $x\in I_t$, with $L_t$ replaced by $L_x$, again yields \eqref{4.3}.
Since we need to  show that $u(t,x)$ points (non-tangentially) outside $L_t$ at $x$, it now suffices to show that 
 \beq\lb{4.31}
u_2(t,x)\ge 0
\eeq
when $\eps'>0$ is small enough.  

We proceed similarly as for $x\in I_t$, but the argument will this time be a little simpler.
Odd symmetry of $\tht$ in both $x_1$ and $x_2$  shows that
\beq\lb{4.34}
u_2(t,x) = \int_{D^+} K_2(x,y) \tht(t,y) dy,
\eeq
where
\beq\lb{4.34a}
K_2(x,y):= \frac{y_1-x_1}{|x-y|^{2+2\alpha}} - \frac{y_1-x_1}{|x-\bar y|^{2+2\alpha}} 
+ \frac{y_1+x_1}{|x-\tilde y|^{2+2\alpha}} - \frac{y_1+x_1}{|x+ y|^{2+2\alpha}},
\eeq
and now we let 
\[
A^\pm_x:=\{y\in D^+\,|\, \pm K_2(x,y)>0\}.
\]
Since $0\le\tht\le 1$ on $D^+$, a lower bound on $u_2(t,x)$ can be obtained by replacing $\tht(t,\cdot)$ in \eqref{4.34} by 1 on $A^-_x\cup L_x$ and by 0 on $A^+_x\setminus L_x$ (so now we just pick $L'_x= L_x$) because $A^-_x$ is now the ``bad'' set for a lower bound.  We will again use a slightly weaker bound,  replacing  $\tht(t,\cdot)$ by 1 on $A^-_x\cup L_x\cup L''_x$ and by 0 on $A^+_x\setminus (L_x\cup L''_x)$, as well as $K_2(x,y)$ by
\[
\frac{y_1-x_1}{|x-y|^{2+2\alpha}} - \frac{y_1-x_1}{|x-\bar y|^{2+2\alpha}}  \qquad(\le\min\{K_2(x,y),0\} \text{ when $y_1\le x_1$})
\]
for all $y\in L''_x:= (0,x_1)\times\bbR^+$ (so this time there will be no $L'''_x$ and no need for Lemma \ref{L.4.1}).  
After expanding the integral from \eqref{4.34} onto $\bbR^2$, this lower bound on $u_2(t,x)$ becomes
\beq\lb{4.34b}
\int_{S_x\cup(-S_x)\cup L''_x}  \frac{y_1-x_1}{|x-y|^{2+2\alpha}} dy - 
\int_{\bar S_x\cup \tilde S_x\cup \bar L''_x}  \frac{y_1-x_1}{|x-y|^{2+2\alpha}} dy,
\eeq
where $S_x:=(A^-_x\cup L_x)\setminus L''_x$.  But since now $A^-_x\subseteq L''_x$ due to
\[
\frac{y_1-x_1}{|x-y|^{2+2\alpha}} \ge \frac{y_1-x_1}{|x-\bar y|^{2+2\alpha}} \qquad\text{and}\qquad 
 \frac{y_1+x_1}{|x-\tilde y|^{2+2\alpha}} \ge \frac{y_1+x_1}{|x+ y|^{2+2\alpha}}
\] 
for all $y\in [x_1,\infty) \times \bbR^+$, it follows from $L_x\cap L''_x=\emptyset$ that $S_x=L_x$.
The regions $L_x,L''_x$, and their reflections are drawn in Figure \ref{F.4.6} (recall that $x_1=x_2=x_*$).

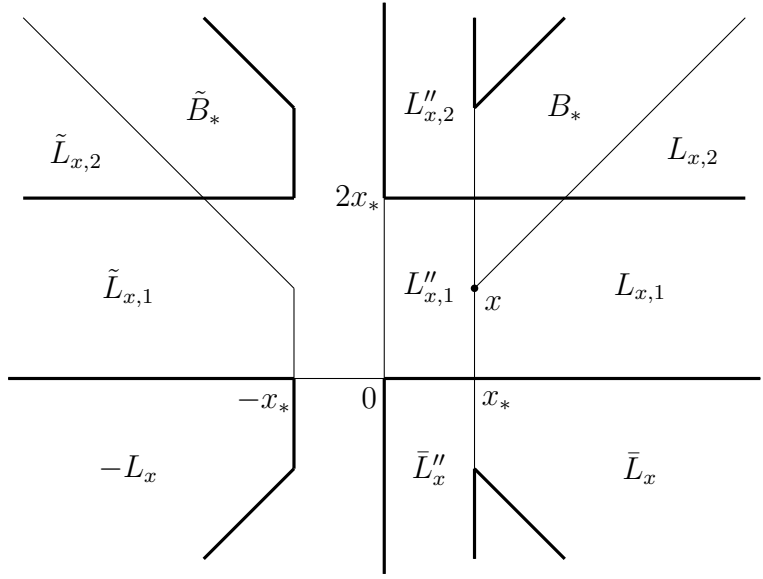
\begin{figure} 
    \pgfdeclarelayer{drawing layer}
    \pgfdeclarelayer{label layer}
    \pgfsetlayers{drawing layer,label layer}
    \begin{tikzpicture}

        \begin{pgfonlayer}{drawing layer}
            \draw (-5,0) to (5,0);
            \draw[very thick] (-5,0) to (-1.2,0);
            \draw[very thick] (0,0) to (5,0);
            \draw (0,-2.6) to (0,5);
            \draw[very thick] (0,-2.6) to (0,0);
            \draw[very thick] (0,2.4) to (0,5);
            
            \draw (1.2,-2.4) to (1.2,4.8);
            \draw[very thick] (1.2,-2.4) to (1.2,-1.2);
            \draw[very thick] (1.2,3.6) to (1.2,4.8);
            \draw[very thick] (1.2,-1.2) to (2.4,-2.4);
            \draw[very thick] (1.2,3.6) to (2.4,4.8);
            \draw (1.2,1.2) to (4.8,4.8);
            \draw[very thick] (0,2.4) to (4.8,2.4);

            \draw[very thick] (-1.2,0) to (-1.2,-1.2);
            \draw[very thick] (-1.2,3.6) to (-1.2,2.4);
            \draw[very thick] (-1.2,-1.2) to (-2.4,-2.4);
            \draw[very thick] (-1.2,3.6) to (-2.4,4.8);            
            \draw (-1.2,1.2) to (-4.8,4.8);
            \draw[very thick] (-1.2,2.4) to (-4.8,2.4);
            \draw (-1.2,0) to (-1.2,1.2);

            \fill (1.2,1.2) circle(0.05);
               \node (x) at (-0.2,-0.25) {$0$};
               \node (x) at (1.45,1) {$x$};
               \node (x) at (1.5,-0.3) {$x_*$};
               \node (x) at (-1.6,-0.3) {$-x_*$};
               \node (x) at (-0.35,2.4) {$2x_*$};
               
               \node (x) at (0.6,3.6) {$L''_{x,2}$};
               \node (x) at (0.6,1.2) {$L''_{x,1}$};
               \node (x) at (2.4,3.6) {$B_*$};
               \node (x) at (3.4,1.2) {$L_{x,1}$};
               \node (x) at (4.1,3) {$L_{x,2}$};
               \node (x) at (0.6,-1.2) {$\bar L''_{x}$};
               \node (x) at (3.4,-1.2) {$\bar L_{x}$};
               
               \node (x) at (-2.4,3.6) {$\tilde B_*$};
               \node (x) at (-3.4,1.2) {$\tilde L_{x,1}$};
               \node (x) at (-4.1,3) {$\tilde L_{x,2}$};
               \node (x) at (-3.4,-1.2) {$- L_{x}$};

        \end{pgfonlayer}

    \end{tikzpicture}
    \caption{Subregions and reflections of $L_x,L''_x$ for $u_2$.} \lb{F.4.6}
\end{figure}

So \eqref{4.34b} becomes
\beq \lb{4.34c}
\int_{L_x\cup(-L_x)\cup L''_x}  \frac{y_1-x_1}{|x-y|^{2+2\alpha}} dy - 
\int_{\bar L_x\cup \tilde L_x\cup \bar L''_x}  \frac{y_1-x_1}{|x-y|^{2+2\alpha}} dy,
\eeq
Because the first integrand at $(y_1,y_2)$ equals the second integrand at $(y_1,2x_2-y_2)$, and vice versa, the integrals over the regions contained in $\bbR\times\bbR^-$  cancel with those over the images of these regions under the mapping $y\mapsto (y_1,2x_2-y_2)$ (in the notation from Figure \ref{F.4.6}, region  $\bar L''_x \cup \bar L_{x}$ cancels with $L''_{x,2}\cup L_{x,2}$ and $-L'_x$ cancels with $\tilde L_{x,1} \cup \tilde L_{x,2}$).  However, this cancellation requires us to add the regions
\[
B_* :=  \{ y\in \bbR^2 \,|\, y_1>x^* \,\,\&\,\, y_2\in(\min\{y_1,2x_*\} ,y_1+2x_*) \}
\]
and $\tilde B_*$ to $L_x$ and $\tilde L_x$, respectively (in fact, then $B_*\cap([1,\infty)\times\bbR)$ and $\tilde B_*\cap((-\infty,-1]\times\bbR)$ will not get cancelled, but their contribution is again only $O(x_*)$ and hence negligible) So we have to subtract their contribution, meaning that \eqref{4.31} will follow if we can show
\beq \lb{4.34d}
\int_{G_*\cup B \cup \tilde B_*}  \frac{y_1-x_1}{|x-y|^{2+2\alpha}} dy - 
\int_{\tilde G_*\cup B_*}  \frac{y_1-x_1}{|x-y|^{2+2\alpha}} dy \ge C x_*
\eeq
for all $x_*\in(0,\eps']$, for any constant $C<\infty$  and small enough ($C$-dependent) $\eps'>0$, where 
\[
G_* := \{ y\in \bbR^2 \,|\, y_1>x^* \,\,\&\,\, y_2\in(0,\min\{y_1,2x_*\}) \}
\]
and  $B:=(0,x_*)\times(0,2x_*)$ (note that the difference between $G_* \cup \tilde G_*$ and $L_{x,1} \cup \tilde L_{x,1}$ from Figure \ref{F.4.6} is a subset of $[\bbR\setminus(-1,1)]\times (0,2x_*)]$ and hence only contributes $O(x_*)$ to \eqref{4.34d}).

We again apply to \eqref{4.34d} the change of variables $z:=\frac{y-x}{x_*}$, but first we note that the mapping $y\mapsto (2x_1-y_1,y_2)$ changes the sign of the integrand, so the integral over $B$ cancels with that over $(x_*,2x_*)\times (0,2x_*)$, which means that we can remove domains $B$ and $G_*\cap[(x_*,2x_*)\times(0,2x_*)]$ from the first integral if we also add the triangle $[(x_*,2x_*)\times (0,2x_*)]\setminus G_*$ to the second.  This and the change of variables now show that \eqref{4.34d} is equivalent to
\[
 \int_{G_0 \cup G^-_0}  \frac{|z_1|}{|z|^{2+2\alpha}} dz - \int_{B_0 \cup B^-_0}  \frac{|z_1|}{|z|^{2+2\alpha}} dz 
 \ge  C x_*^{2\alpha},
\]
where the sets $G_0, G^-_0, B_0, B^-_0$ are from the following lemma (also drawn in Figure \ref{F.4.8}).  The last estimate now follows for all small enough $\eps'>0$ and all $x_*\in(0,\eps]$ by our last key Lemma~\ref{L.4.3}, which we prove after the proof of Lemma \ref{L.4.2}.

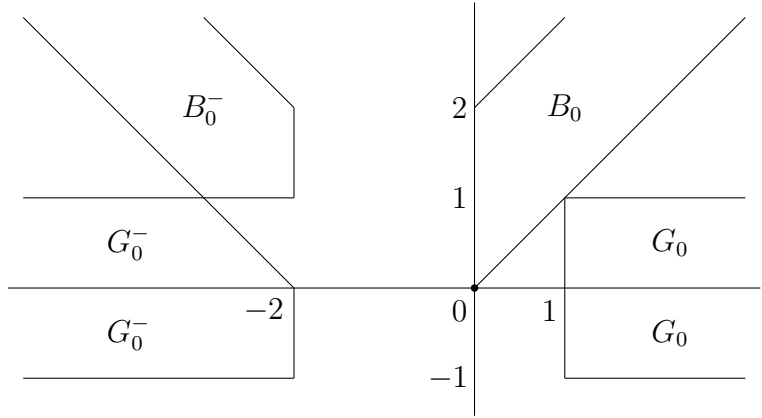
\begin{figure} 
    \pgfdeclarelayer{drawing layer}
    \pgfdeclarelayer{label layer}
    \pgfsetlayers{drawing layer,label layer}
    \begin{tikzpicture}

        \begin{pgfonlayer}{drawing layer}
            \draw (-5,1.2) to (5,1.2);
            \draw (1.2,-0.5) to (1.2,5);
            \draw(1.2,3.6) to (2.4,4.8);
            \draw (1.2,1.2) to (4.8,4.8);
            \draw (2.4,2.4) to (4.8,2.4);

            \draw (-1.2,3.6) to (-1.2,2.4);
            \draw (-1.2,3.6) to (-2.4,4.8);            
            \draw (-1.2,1.2) to (-4.8,4.8);
            \draw (-1.2,2.4) to (-4.8,2.4);
            \draw (-1.2,0) to (-1.2,1.2);
            \draw (-4.8,0) to (-1.2,0);
            \draw (4.8,0) to (2.4,0);
            \draw (2.4,2.4) to (2.4,0);

            \fill (1.2,1.2) circle(0.05);
               \node (x) at (1,0.9) {$0$};
               \node (x) at (-1.6,0.9) {$-2$};
               \node (x) at (2.2,0.9) {$1$};
               \node (x) at (1,3.6) {$2$};
               \node (x) at (1,2.4) {$1$};               
               \node (x) at (0.85,0) {$-1$}; 
               
               \node (x) at (2.4,3.6) {$B_0$};
               \node (x) at (-2.4,3.6) {$B^-_0$};
               \node (x) at (-3.4,1.8) {$G^-_0$};
               \node (x) at (-3.4,0.6) {$G^-_0$};
               \node (x) at (3.8,1.8) {$G_0$};
               \node (x) at (3.8,0.6) {$G_0$};
               
        \end{pgfonlayer}

    \end{tikzpicture}
    \caption{The regions $G_0,G^-_0, B_0,B^-_0$ for $u_2$.} \lb{F.4.8}
\end{figure}

\begin{lemma} \lb{L.4.3}
If $\alpha\in(0,\frac 14]$ and we let 
\begin{align*}
G_0 & := (1,\infty)\times(-1,1) \\
G^-_0 & := \{ y\in \bbR^2 \,|\, y_1<-2 \,\,\&\,\, y_2\in(-1,\min\{-y_1-2,1\}) \}, \\
B_0 & :=\{ y\in \bbR^2 \,|\, y_1>0 \,\,\&\,\, y_2\in(y_1 ,y_1+2) \}, \\
B^-_0 & :=\{ y\in \bbR^2 \,|\, y_1<-2 \,\,\&\,\, y_2\in(\max\{-y_1-2,1\} ,-y_1) \},
\end{align*}
then
\beq \lb{4.40}
 \int_{G_0 \cup G^-_0}  \frac{|x_1|}{|x|^{2+2\alpha}} dx - \int_{B_0 \cup B^-_0}  \frac{|x_1|}{|x|^{2+2\alpha}} dx >0 .
\eeq
\end{lemma}

So \eqref{4.3} and \eqref{4.31} hold whenever $\eps'>0$ is small enough, which concludes the proof of Theorem \ref{T.1.1}(ii) as explained before \eqref{4.3}.  

\bigskip\noindent
{\bf Proofs of the two velocity estimates.}  It remains to prove Lemmas \ref{L.4.2} and \ref{L.4.3}.

\begin{proof}[Proof of Lemma \ref{L.4.2}]
Let $I(b)$ be $2\alpha$ times the value inside the infimum in \eqref{4.7} when $G_b^-$ is replaced by $G_b^*:=G_b^-\cap \{x\in\bbR^2\,|\, x_1<-x_2\}$, which turns out to simplify the proof significantly.  This does not change the integration  region when $b=1$, while for $b<1$ it decreases the value inside the infimum (after being multiplied by $2\alpha$, which will also be convenient).  
Note that since $I$ is  continuous on $[0,1]$,  it suffices to consider $b\in(0,1]$.

If now $G_b^+,B_b^+$ are the reflections of 
$G_b^*,B_b^-$ across the $x_1$ axis, then $G_b,G_b^+,B_b^+\subseteq \bbR\times\bbR^+$ and
\[
I(b)=  2\alpha \int_{G_b\cup G_b^+} \frac{x_2}{|x|^{2+2\alpha}} dx - 2\alpha  \int_{B_b^+} \frac{x_2}{|x|^{2+2\alpha}} dx.
\]
  When $0\le a\le c$,  then direct integration gives
\beq \lb{4.8}
2\alpha \int_a^c \frac{x_2}{|(x_1,x_2)|^{2+2\alpha}} dx_2 = \frac 1{(x_1^2+a^2)^{\alpha}} - \frac 1{(x_1^2+c^2)^{\alpha}},
\eeq
and for $0\le r\le s$ we also have
\beq \lb{4.9}
\int_{r}^{s} \frac {dx_1} {(x_1^2+a^2)^{\alpha}}  = 
\begin{cases}
a^{1-2\alpha} \left( f \left( \frac sa \right) - f \left( \frac ra \right) \right) & a>0, \\
\frac { 1 } {1-2\alpha} \left( s^{1-2\alpha} - r^{1-2\alpha} \right) & a=0,
\end{cases}
\eeq
where
\beq \lb{4.10}
f(s):=\int_0^s \frac {dq}{(q^2+1)^{\alpha}}.
\eeq
We will use \eqref{4.8} and \eqref{4.9} repeatedly below, without saying so explicitly.
We also note that from $0\le (q^2)^{-\alpha} - (q^2+1)^{-\alpha} \le \alpha q^{-2-2\alpha}$ we find that there exists
\beq \lb{4.10a}
\mu_\alpha:= \lim_{s\to\infty} \left( \frac{s^{1-2\alpha}}{1-2\alpha}-f(s) \right)>0,
\eeq
which then satisfies
\beq \lb{4.10c}
\mu_\alpha \le \frac 1{1-2\alpha}-f(1)+\int_1^\infty \frac \alpha {q^{2+2\alpha}}dq \le \frac 1{1-2\alpha}-f(1)+ \frac \alpha{1+2\alpha}.
\eeq
as well as
\beq \lb{4.10b}
\lim_{s\to\infty} \left( \gamma^{2\alpha-1} f(\gamma s) - \frac{s^{1-2\alpha}}{1-2\alpha} \right) =- \gamma^{2\alpha-1}\mu_\alpha
\eeq
 for any $\gamma>0$.

First, we easily get
\beq \lb{4.11}
2\alpha  \int_{B_b^+} \frac{x_2}{|x|^{2+2\alpha}} dx = 2\int_{0}^1 \left( \frac 1{x_1^{2\alpha}} - \frac 1{(x_1^2+b^2)^{\alpha}} \right) dx_1
= \frac 2{1-2\alpha} - 2b^{1-2\alpha} f\left(\frac 1b \right).
\eeq
Next,  $x_1^2+(x_1-2)^2= 2((x_1-1)^2+1)$ and \eqref{4.10b} yield
\beq \lb{4.12}
\begin{split}
2\alpha  \int_{G_b} \frac{x_2}{|x|^{2+2\alpha}} dx & = 
\int_0^2  \left( \frac 1 {x_1^{2\alpha}} - \frac{2^{-\alpha}} {x_1^{2\alpha}} \right) dx_1
+ \int_2^\infty  \left( \frac {2^{-\alpha}} {((x_1-1)^2+1)^{\alpha}} - \frac{2^{-\alpha}} {x_1^{2\alpha}} \right) dx_1 \\
& = \frac{2^{1-2\alpha} } {1-2\alpha} - 2^{-\alpha} f(1) - 2^{-\alpha} \mu_\alpha
\end{split}
\eeq
(here we also used that $\lim_{s\to\infty} f'(s)=0$).
Finally,  $x_1^2+(x_1+2b)^2= 2((x_1+b)^2+b^2)$ 
yields
\[
\begin{split}
2\alpha  \int_{G_b^+} \frac{x_2}{|x|^{2+2\alpha}} dx  & =  
\int_{-b}^{b}   \left(  \frac 1{(x_1^2+b^2)^{\alpha}}  - \frac{2^{-\alpha}} {((x_1+b)^2+b^2)^{\alpha}} \right) dx_1 
 + \int_{b}^\infty  \left( \frac {2^{-\alpha}}{x_1^{2\alpha}}  -\frac{2^{-\alpha}} {((x_1+b)^2+b^2)^{\alpha}}  \right) dx_1 \\
 & =  2b^{1-2\alpha} f(1)  - \frac{ 2^{-\alpha} b^{1-2\alpha}} {1-2\alpha} 
 + 2^{-\alpha}b^{1-2\alpha} \mu_\alpha,
\end{split}
\]
where we also used \eqref{4.10b}.
Hence we see that
\begin{align*}
I(b)=   \frac {2^{1-2\alpha}-2}{1-2\alpha} -2^{-\alpha} (f(1) +  \mu_\alpha) +
2b^{1-2\alpha} \left( f\left(\frac 1b \right) +  f(1) - \frac{ 2^{-1-\alpha} } {1-2\alpha} + 2^{-1-\alpha} \mu_\alpha \right) 
\end{align*}
for all $b\in(0,1]$, and so also
\beq \lb{4.16}
I(1)= (4-2^{-\alpha}) f(1) -\frac {2+2^{-\alpha}-2^{1-2\alpha}}{1-2\alpha} .
\eeq

If we let $J(c):= \frac 1{2}I(\frac 1c)$ for $c\in[1,\infty)$, then
\[
J'(c)= (1-2\alpha)c^{2\alpha-2} \left(  \frac c{(1-2\alpha)(c^2+1)^\alpha} -f(c) - f(1) + \frac{ 2^{-1-\alpha} } {1-2\alpha}  - 2^{-1-\alpha} \mu_\alpha \right) =: (1-2\alpha)c^{2\alpha-2} g(c).
\]
Since
\[
g'(c) = \frac 1{(1-2\alpha)(c^2+1)^\alpha} \left(  1- \frac {2\alpha c^2}{c^2+1}-(1-2\alpha) \right)=  \frac {2\alpha}{(1-2\alpha)(c^2+1)^{1+\alpha}}\ge 0,
\]
we see that
\beq \lb{4.50}
\min_{c\in[1,\infty)} g(c)=g(1)= \frac {3\cdot 2^{-1-\alpha}}{1-2\alpha} - 2 f(1) - 2^{-1-\alpha} \mu_\alpha.  
\eeq
When $g(1)\ge 0$ (which happens for $\alpha$ not too small), then $J$ attains its minimum at $c=1$, so we also have $\min_{b\in(0,1]} I(b)=I(1)$.  If instead $g(1)<0$, then
\[
\inf_{c\in[1,\infty)} J(c)\ge J(1)+\int_1^\infty (1-2\alpha) c^{2\alpha-2} g(1) dc = J(1)+g(1),
\]
and so $\inf_{b\in(0,1]} I(b)\ge I(1)+2 g(1)$.  This means that
\[
\inf_{b\in(0,1]} I(b)\ge \min\{I(1),I(1)+2 g(1)\}.
\]

So it remains to estimate these two values.  From \eqref{4.16}, \eqref{4.50}, and \eqref{4.10c} we obtain
\[
2^\alpha(I(1)+2g(1)) = \frac {2+2^{1-\alpha} -2^{1+\alpha} }{1-2\alpha} - \mu_\alpha - f(1)  \ge \frac {1+2^{1-\alpha} -2^{1+\alpha} }{1-2\alpha}- \frac \alpha{1+2\alpha}.
\]
Since $\frac \alpha{1+2\alpha}\le \frac 16$ and $1+2^{1-\alpha} -2^{1+\alpha}\ge \frac 3{10}$ when $\alpha\in[0,\frac 14]$, it follows that
\[
I(1)+2g(1) \ge \frac {2^{3/4}}{15} \ge \frac 1{10}.
\] 
Hence it suffices to show that
$I(1)\ge \frac 1{20}$, which thanks to   \eqref{4.16} will follow from
\[
 f(1) \ge \frac {2+2^{-\alpha}-2^{1-2\alpha}}{(1-2\alpha)(4-2^{-\alpha})} +\frac 1{60}.
\]
  Since the numerator and denominator of the first fraction on the right-hand side are increasing and decreasing in $\alpha$, respectively, and $f(1)$ is decreasing in $\alpha$, it suffices to prove this for $\alpha=\frac 14$, when the inequality becomes
\beq \lb{4.17}
f(1)\ge \frac{4+2^{3/4}-2^{3/2} }{4-2^{-1/4}} + \frac 1{60}.
\eeq
The first fraction on the right is less than 0.9033 and $f(1)$ for $\alpha=\frac 14$ can be numerically bounded below by 0.9374, which yields \eqref{4.17}.  Alternatively, since $(q^2+1)^{-1/4}$ is concave on $[0,\sqrt{2/3}]\ni \frac 45$, and we have $((\frac 45)^2+1)^{-1/4}\ge \frac{22}{25}$ and $2^{-1/4}\ge \frac{21}{25}$, it follows that 
\[
f(1)\ge \frac 25 \left(1+ \frac{22}{25}\right)+\frac 15\, \frac{21}{25} = 0.92,
\]
and we conclude by $0.92\ge 0.9033 + \frac 1{60}$.
%
%
\end{proof}

\begin{proof}[Proof of Lemma \ref{L.4.3}]
If $G^+_0,B^+_0$ are reflections of $G^-_0,B^-_0$ across the $x_2$ axis, then the left-hand side of \eqref{4.40} times $2\alpha$ is
\[
I:=2\alpha \int_{G_0\cup G^+_0} \frac{x_1}{|x|^{2+2\alpha}} dx - 2\alpha  \int_{ B_0 \cup B^+_0} \frac{x_1}{|x|^{2+2\alpha}} dx
\]
and all 4 sets are contained in $\bbR^+\times\bbR$.
When $0\le a\le c$,  then direct integration gives
\beq \lb{4.41}
2\alpha \int_a^c \frac{x_1}{|(x_1,x_2)|^{2+2\alpha}} dx_1 = \frac 1{(x_2^2+a^2)^{\alpha}} - \frac 1{(x_2^2+c^2)^{\alpha}},
\eeq
and with $f$ from \eqref{4.10} we again  have
\beq \lb{4.42}
\int_{r}^{s} \frac {dx_2} {(x_2^2+a^2)^{\alpha}}  = 
\begin{cases}
a^{1-2\alpha} \left( f \left( \frac sa \right) - f \left( \frac ra \right) \right) & a>0, \\
\frac { 1 } {1-2\alpha} \left( s^{1-2\alpha} - r^{1-2\alpha} \right) & a=0
\end{cases}
\eeq
(we will again use these repeatedly below).

We now see that 
\[
2\alpha  \int_{B_0} \frac{x_1}{|x|^{2+2\alpha}} dx  = \frac{2^{1-2\alpha} } {1-2\alpha} - 2^{-\alpha} f(1) - 2^{-\alpha} \mu_\alpha
\]
because this is the same integral as for the region $G_b$ in the previous proof.  Next, 
\begin{align*}
2\alpha  \int_{B_0^+} \frac{x_1}{|x|^{2+2\alpha}} dx  & = 
 \int_1^2  \left( \frac{1}{(x_2^2+2^2)^\alpha} -  \frac{2^{-\alpha}} {((x_2+1)^2+1)^\alpha} \right) dx_2
 + \int_2^\infty  \left( \frac{2^{-\alpha}} {x_2^{2\alpha}} -  \frac{2^{-\alpha}} {((x_2+1)^2+1)^\alpha} \right) dx_2 \\
 &= 2^{1-2\alpha}  \left( f(1)-f \left( \frac 12 \right) \right) + 2^{-\alpha} f(2) - \frac {2^{1-3\alpha}}{1-2\alpha} + 2^{-\alpha} \mu_\alpha
\end{align*}
and
\[
2\alpha  \int_{G_0} \frac{x_1}{|x|^{2+2\alpha}} dx  = 2 \int_0^1 \frac{dx_2}{(x_2^2+1)^\alpha} = 2f(1).
\]
Finally, from $x_2^2+(x_2+2)^2=2((x_2+1)^2+1)^2$ we obtain 
\[
2\alpha  \int_{G_0^+} \frac{x_1}{|x|^{2+2\alpha}} dx  = 
 \int_0^1 \frac{dx_2}{(x_2^2+2^2)^\alpha} +  \int_0^1 \frac{2^{-\alpha}} {((x_2+1)^2+1)^\alpha} dx_2
 = 2^{1-2\alpha} f \left( \frac 12 \right) + 2^{-\alpha} (f(2)-f(1)).
\]
Hence
\[
I =(2-2^{1-2\alpha}) f(1) +2^{2-2\alpha} f \left( \frac 12 \right) - \frac {2^{1-2\alpha}(1-2^{-\alpha})} {1-2\alpha} .
\]
Since
\[
2^{2-2\alpha} f \left( \frac 12 \right) > 2^{3/2} \frac 12 \left( \frac 45 \right)^{1/4} > 1 > 4 (1-2^{-1/4}) > \frac {2^{1-2\alpha}(1-2^{-\alpha})} {1-2\alpha}
\]
for $\alpha\in(0,\frac 14]$, \eqref{4.40} follows.
\end{proof}




\begin{thebibliography}{99}



\bibitem{Blu}
W. Blumen,
\it Uniform potential vorticity flow. Part I: Theory of wave interactions and two-dimensional turbulence,
\rm J. Atmos. Sci. {\bf 35} (1978), 774--783.

\bibitem{BSV}
T. Buckmaster, S. Shkoller, and V. Vicol, 
\it Nonuniqueness of weak solutions to the SQG equation
\rm Comm. Pure Appl. Math. {\bf 72} (2019), 1809--1874.

%
%
%
%

\bibitem{CCW}
D. Chae, P. Constantin, and J. Wu,
\it Inviscid models generalizing the two-dimensional Euler and the surface quasi-geostrophic equations,
\rm Arch. Ration. Mech. Anal. {\bf 202} (2011), 35--62. 



\bibitem{CheHou}
J. Chen and T.Y. Hou,
\it Stable nearly self-similar blowup of the 2D Boussinesq and 3D Euler enquations with smooth initial data I: Analysis,
\rm preprint, arXiv:2210.07191.

\bibitem{CheHou2}
J. Chen and T.Y. Hou,
\it Stable nearly self-similar blowup of the 2D Boussinesq and 3D Euler enquations with smooth initial data II: Rigorous Numerics,
\rm preprint, arXiv:2305.05660.
  

\bibitem{CIN}
P.~Constantin, M. Ignatova, and H.Q. Nguyen,
\it Inviscid limit for SQG in bounded domains,
\rm SIAM J. Math. Anal. {\bf 50} (2018),  6196--6207. 

\bibitem{CIW} 
P.~Constantin, G.~Iyer, and  J.~Wu,
\it Global regularity for a modified critical dissipative quasi-geostrophic equation,
\rm Indiana Univ. Math. J. {\bf 57} (2008), 2681--2692.

\bibitem{CMT} 
P.~Constantin, A.~Majda, and E.~Tabak, 
\textit{Formation of strong fronts in the 2D quasi-geostrophic thermal active scalar},
Nonlinearity {\bf 7} (1994), 1495--1533.

\bibitem{ConNgu} 
P.~Constantin and H.Q. Nguyen,
\it Local and global strong solutions for SQG in bounded domains,
\rm Phys. D {\bf 376/377} (2018), 195--203. 

\bibitem{ConNgu2} 
P.~Constantin and H.Q. Nguyen,
\it Global weak solutions for SQG in bounded domains,
\rm Comm. Pure Appl. Math. {\bf 71} (2018), 2323--2333. 




%
%
%

%


\bibitem{CorMar}
D. C\' ordoba and L. Mart\' inez-Zoroa,
\it Non existence and strong ill-posedness in $C^k$ and Sobolev spaces for SQG,
\rm Adv. Math {\bf 407} (2022), 108570.



%
%
%
%
%
%

\bibitem{Elgindi}
T. Elgindi, 
\it Finite-time singularity formation for $C^{1,\alpha}$ solutions to the incompressible Euler equations on $\bbR^3$,
\rm Ann. of Math. (2) {\bf 194} (2021),  647--727.

%

\bibitem{GanPat}
F. Gancedo and N. Patel,
\it On the local existence and blow-up for generalized SQG patches,
\rm Ann. PDE {\bf 7} (2021),  Paper No. 4, 63 pp.

%

\bibitem{HanZla}
Z.~Han and A.~Zlato\v s,
{\it Euler equations on general planar domains},
Ann. PDE {\bf 7} (2021), Article 20, 31pp.

\bibitem{HanZla2}
 Z.~Han and A.~Zlato\v s,
{\it Uniqueness of positive vorticity solutions to the 2D Euler equations on singular domains},
Arch. Ration.~Mech.~Anal. {\bf 247} (2023), Article 84, 22pp.





\bibitem{Holder} 
E.~H\"older, 
\it \"Uber die unbeschr\"ankte Fortsetzbarkeit einer stetigen ebenen Bewegung in einer
unbegrenzten inkompressiblen Fl\"ussigkeit, 
\rm Math. Z. {\bf 37} (1933), 727--738.


\bibitem{JeoZla}
J. Jeon and A.~Zlato\v s,
{\it An improved regularity criterion and absence of splash-like singularities for g-SQG patches}, 
Anal.~PDE, to appear.

%
%
%
%
%


\bibitem{KS}
A. Kiselev and V. \v Sver\' ak, 
\it Small scale creation for solutions of the incompressible two dimensional Euler equation, 
\rm Ann. of Math. (2) {\bf 180} (2014), 1205--1220.

\bibitem{KYZ}
A. Kiselev, Y. Yao, and A. Zlato\v{s},
\it Local regularity for the modified SQG patch equation,
\rm Comm. Pure and Appl. Math. {\bf 70} (2017), 1253--1315.

\bibitem{KRYZ}
A.~Kiselev, L.~Ryzhik, Y.~Yao, and A.~Zlato\v s, 
{\it Finite time singularity for the modified SQG patch equation}, 
Ann.~of~Math.~(2) {\bf 184} (2016), 909--948.

 
%




%
%

\bibitem{Nguyen} 
H.Q. Nguyen,
\it Global weak solutions for generalized SQG in bounded domains,
\rm Anal. PDE {\bf 11} (2018), 1029--1047. 



\bibitem{Ped}
J. Pedlosky, \it Geophysical Fluid Dynamics, \rm Springer, New York, 1987.

\bibitem{PHS}
R.T. Pierrehumbert, I.M. Held, and K.L. Swanson,
\it  Spectra of local and nonlocal two-dimensional turbulence,
\rm  Chaos, Solitons Fractals {\bf 4} (1994), 1111--1116.


\bibitem{Resnick}
S. Resnick, 
\it Dynamical problems in nonlinear advective partial differential equations,
\rm Ph.D. thesis, University of Chicago,  1995.


\bibitem{Smith}
K.S. Smith, G. Boccaletti, C.C. Henning, I. Marinov, C.Y. Tam, I.M. Held, and G.K. Vallis,
\it Turbulent diffusion in the geostrophic inverse cascade,
\rm J. Fluid Mech. {\bf 469} (2002), 13--48.


\bibitem{Wolibner} 
W. Wolibner, 
\it Un theor\`eme sur l'existence du mouvement plan d'un fluide parfait, homog\`ene,
incompressible, pendant un temps infiniment long, 
\rm Mat. Z., {\bf 37} (1933), 698--726.

\bibitem{Wu}
J. Wu,
\it Solutions of the 2D quasi-geostrophic equation in H\" older spaces
\rm  Nonlinear Anal. {\bf 62} (2005), 579--594. 

%
%

\bibitem{Yudth} 
V.I.~Yudovich, 
\it Non-stationary flows of an ideal incompressible fluid, 
\rm Zh. Vych. Mat. {\bf 3} (1963), 1032--1066.

\bibitem{ZlaEulerexp}
A.~Zlato\v s,
\it Exponential growth of the vorticity gradient for the Euler equation on the torus,
\rm Adv. Math. {\bf 268} (2015), 396--403.


\end{thebibliography}
\end{document}